\documentclass[11pt, makeidx]{amsart}

\usepackage{enumerate}
\usepackage{amssymb,amsfonts,amsthm}
\usepackage{srcltx}
\usepackage{float}
\usepackage[OT2,OT1]{fontenc}
\newcommand\cyr{%
\renewcommand\rmdefault{wncyr}%
\renewcommand\sfdefault{wncyss}%
\renewcommand\encodingdefault{OT2}%
\normalfont
\selectfont}
\DeclareTextFontCommand{\textcyr}{\cyr}
\usepackage{color}
\makeindex

\makeatletter
\def\sssub{\@startsection{paragraph}{4}}

\renewcommand\paragraph{\@startsection{paragraph}{4}{\z@}{1.25ex}{0.0001pt}{\normalfont\normalsize\em}}

\makeatother

\numberwithin{paragraph}{subsubsection}

\newcommand{\ee}{\end{enumerate}}
\newcommand{\beq}{\begin{equation}}
\newcommand{\eeq}{\end{equation}}

\newcommand{\Z}{{\mathbb{Z}}}
\newcommand{\N}{{\mathbb{N}}}
\newcommand{\inp}{{\mathrm{inp}}}
\newcommand{\Sym}{{\mathrm{Sym}}}
\newcommand{\Add}{{\mathrm{Add}}}
\newcommand{\Sub}{{\mathrm{Sub}}}
\newcommand{\Aut}{{\mathrm{Aut}}}
\newcommand{\la}{\langle}
\newcommand{\ra}{\rangle}
\newcommand{\iv}{^{-1}}

\newtheorem{theorem}{Theorem}[section]

\newtheorem{lemma}[theorem]{Lemma}

\theoremstyle{remark}
\newtheorem{remark}[theorem]{Remark}

\begin{document}

\title{Algorithmically complex residually finite groups}
\author[O. Kharlampovich]{O. Kharlampovich$^1$}
\author[A. Myasnikov]{A. Myasnikov$^2$}
\author[M. Sapir]{M. Sapir$^{3}$}
\thanks{$^1$ Partially supported by NSF grant DMS-0700811, $^2$ Partially supported by NSF
grants DMS-0700811 and DMS-0914773. $^{3}$ Partially supported by NSF grant DMS-0700811 and by BSF (U.S.A.-Israel) grant 2010295.}

\address{Olga Kharlampovich, Department of Mathematics and Statistics, Hunter College, City University of New York, New York, NY, 10065, U.S.A.}
\email{okharlampovich@gmail.com}

\address{Alexei Myasnikov, Stevens Institute of Technology, Hoboken, NJ, 07030 U.S.A.}
\email{amiasnik@stevens.edu}

\address{Mark Sapir, Department of Mathematics, Vanderbilt University, Nashville, TN 37240,
 U.S.A.} \email{m.sapir@vanderbilt.edu}

\maketitle \begin{abstract} We construct the first examples of
algorithmically complex finitely presented residually finite groups, the
first examples of finitely presented residually finite groups with
arbitrarily high (recursive) Dehn functions, and arbitrary large depth
functions. The groups are solvable of class 3. We also prove that the
universal theory of finite solvable of class 3 groups is undecidable.
\end{abstract}

\tableofcontents

\section{Introduction}

\subsection{The problem and previous approaches for a solution}

It is well known that finitely presented residually finite groups are much simpler algorithmically than arbitrary finitely presented groups. For example, the word problem in every such group is decidable. Moreover the most ``common" residually finite groups, the linear groups over fields, are algorithmically ``tame": the word problem in any linear group is decidable in polynomial time and even log-space \cite{LZ77}. Dehn functions are important witnesses of the complexity of the word problem. Although it is well known that groups with solvable word problem can have very large Dehn functions \cite{Madlener-Otto,SBR}, no examples of finitely presented residually finite groups with superexponential Dehn  function are known. Thus one  of the main open problems in this area is how large could the Dehn function of a residually finite finitely presented group be. The question was known since early 90s. It was open for so long because all known methods to construct algorithmically hard  groups produced either non-residually finite groups or groups where the question about their residual finiteness is very difficult. Not much is known even for  linear groups (Steve Gersten asked \cite{Ger, GR} if there exists a uniform upper bound for Dehn functions of linear groups). Let us briefly discuss the previous attempts to solve the problem and the reasons why these methods did not work.

\subsubsection{Method 1. Known groups with large Dehn functions}\label{ss:kn} One could hope that some of the known finitely presented groups with very large Dehn function may turn out to be residually finite, which would shed some light on how to   produce residually finite groups with even larger Dehn functions. Unfortunately, this is not the case, all these groups are  non-residually finite. For example, the Dehn function of the one relator group $G_{(1,2)}  = \langle a,b \mid b^{-1} a^{-1} b a b^{-1} a b = a^2  \rangle$ introduced by Baumslag   in \cite{Baumslag:1969} is bigger than any iterated exponent (see Gersten \cite{G1}). Platonov \cite{P} proved that it is equivalent to the function $\exp^{(\log 2n)}(1)$, where $\exp^{(m)}(x)$ is the function defined by $\exp^{(0)}(n)=n$ and $\exp^{(k+1)}(n)=\exp(\exp^{(k)}(n)))$. However, $G_{(1,2)}$ is not residually finite (and in fact has very few finite quotients) \cite{Baumslag:1969}. Furthermore, the word problem in $G_{(1,2)}$ is in polynomial time \cite{MUW}.

\subsubsection{Method 2. Using subgroups with very large distortion}  Consider a finitely presented group $G$ and a ``badly" distorted finitely generated subgroup $H$. Let $T = \langle G,t \mid t^{-1}ht = h \  (h \in H)\rangle$ be the HNN extension of $G$ where the free letter $t$ centralizes $H$. It was noticed by Bridson and H\"afliger \cite[Theorem 6.20.III]{BH} that the Dehn function of $T$ is at least as large as  the distortion function of $H$ in $G$. The following result puts some limitations on this method of constructing complicated residually finite groups.
\begin{lemma} \label{le:constrT}
If the group $T$ is residually finite then  $H$ is closed in the pro-finite topology of $G$.
\end{lemma}

\proof In the notation above, suppose  $T$ is residually finite and  $u\in G\setminus H$.  Then $w=[u,t]\ne 1$ in $T$ (by the standard properties of HNN extensions). Hence  there exists a homomorphism $\phi$ from $T$ onto a finite group $T_w$ such that $\phi(w)\ne 1$. But this implies $\phi(u)\not \in \phi(H)$ (since every element of $\phi(H)$ commutes with $\phi(t)$). Hence there exists a normal subgroup of finite index $N<G$ such that $u$ does not belong to $NH$. In other words $H$ must be closed in the pro-finite topology of $G$.
\endproof

Consider the following typical examples of residually finite groups $G$ with highly distorted subgroups $H$.  The first one is Wise's version \cite{WiseRF} of Rips' construction \cite{Rips}  which for every finitely presented group $Q$ gives a finitely presented residually finite  small cancellation group $G$ with a short exact sequence $$1\to N \to G \to Q\to 1$$ where $N$ finitely generated. It is easy to see that the distortion function of $N$ in $G$ is at least as bad as the Dehn function in $Q$, so choosing $Q$ properly one can get a finitely presented residually finite group $G$ with a highly distorted subgroup $N$.  Now, the subgroup $N$ is normal in the HNN extension $T$. So it is closed in the pro-finite topology of $T$ only if $Q=G/N$ is residually finite. By Lemma \ref{le:constrT},  $T$ can be residually finite only if $Q$ is residually finite. In other words to construct a complicated finitely presented residually finite group $T$ one has to have the initial group $Q$ complicated, finitely presented and residually finite as well.

The second example is the standard Mikhailova construction. In this case    highly distorted subgroups of the direct product of two free groups $F_2\times F_2$ can be obtained as equalizers of two homomorphisms $\phi_1\colon E_1\to M$ and $\phi_2\colon E_2\to M$ where $E_1, E_2$ are finitely generated subgroups of $F_2$ and $M$ is finitely presented (see Sections \ref{ss:qn}, \ref{s:dis} below). But by Remark \ref{r:m} below the equalizer is closed in the pro-finite topology only if $M$ is residually finite. Thus as in the previous example, in order to construct an algorithmically complex residually finite finitely presented group using Mikhailova's construction and the HNN extensions as above, one  needs to have already  a finitely presented residually finite algorithmically complex group $M$.

The third example is Cohen's construction of highly distorted subgroups employing the modular machines \cite{Cohen}. One can also prove that in that construction the subgroup will be pro-finitely closed only if the modular machine is very easy.

One can also try to use the hydra groups \cite{DR, BDR} to construct HNN extensions as above with Dehn functions bigger than any prescribed Ackermann function. The question of whether these groups are residually finite is open (most probably the answer is negative because the distorted subgroup should not be closed in the pro-finite topology, but this needs a proof).

\subsubsection{Method 3. Boone-Novikov constructions}\label{ss:1} One of the standard ways to produce algorithmically complicated groups is by simulating Turing machines using free constructions (HNN extensions and amalgamated products) which goes back  to the seminal papers by Boone and Novikov (see, for example, \cite{Rotman}). There are currently many versions of that construction (for a recent survey see \cite{SapBMS}). But in fact, it can be shown that for each known version of the proof of Boone-Novikov theorem using free constructions, even for easy Turing machines the corresponding group is non-residually finite. This is, for instance, the main idea of the example in \cite{KSN2A}. Here is an even easier example. Let $G=G(M)$ be a group constructed by any of these constructions. Then for every input word $u$ of the Turing machine $M$ there exists a word $w=w(u)$ obtained by inserting some copies of $u$ in $w(\emptyset)$, so that $u$ is accepted by $M$ if and only if $w(u)=1$ in $G$. Now consider $M$ that accepts a word $a^n$ if and only if $n\ne 0$ (that machine is actually one of the basic building blocks in \cite{SBR}). Then $w(a^n)=1$ in $G$ if and only if $n\ne 0$. Suppose that there exists a homomorphism $\phi$ onto a finite group $H$ that separates $w(\emptyset)=w(a^0)$ from 1. Then $\phi(a)$ has finite order, say, $s$, in $H$. Therefore $\phi(w(\emptyset))=\phi(w(a^0))=\phi(w(a^s))=\phi(1)=1$, a contradiction. Hence $G(M)$ is not residually finite.

\subsubsection{Method 4. Residually finite groups obtained by free constructions} In general, the question about residual finiteness of free constructions is very difficult. Currently there are only two large classes of groups where the question was settled: these are ascending HNN extensions of free groups \cite{BS05} and certain groups acting ``nicely" on CAT(0)-cubical complexes including small cancelation groups (see the recent work of Wise,  for example, \cite{HW} and references therein). All these groups have easy word problem and uniformly bounded Dehn functions. The reason for the lack of more examples is that groups obtained by free constructions from ``nice" groups contain a lot of extra elements and it is not at all clear how to separate these elements from 1 by homomorphisms onto finite groups. In the two cases when it could be done, it was possible to reformulate the problem in the language of algebraic geometry and geometric topology, respectively.

\subsubsection{Method 5. Other groups with complicated word problem} There are several other constructions of groups with complicated word problem but each of these also almost always produce non-residually finite groups. For example, the group in \cite{MT} is based on the R. Thompson group $V$ which is infinite and simple (hence not residually finite).

\subsubsection{The main results of the paper.}\label{ss:main} In this paper, we construct finitely presented residually finite groups with arbitrary complex word problem, and also easy word problem but arbitrary high (of course recursive) Dehn function, and arbitrary high depth function. We also give applications of these results to the questions about solvability of the universal theory of finite solvable groups.

Our constructions are based on simulating Minsky machines in groups. Surprisingly, the algebraic structure of the group not only depends on the construction itself, but also heavily depends on the computational properties of the machines we simulate: for the group to be  residually finite the Minsky machines should be  sym-universally halting, that is their transition graphs are vertex-disjoint unions of finite trees.

\subsubsection{ What next?}  We expect the approach used in this paper to be useful in solving other problems that are still open. For example, the residually finite version of the Higman embedding theorem would be very desirable. It is known \cite{Meskin, Dyson,Grig,KhoussM} that a finitely generated recursively presented residually finite group may have undecidable word problem, and hence cannot be embedded into a finitely presented residually finite group. But it is not known whether every unsolvability of the word problem is the only obstacle for such an embedding. Thus it would be very interesting to find out whether every finitely generated residually finite group with solvable word problem embeds into a finitely presented residually finite group. Note that usually a version of Boone-Novikov theorem precedes a version of Higman theorem, hence we can consider this paper as a step toward the residually finite version of Higman's theorem.

\subsection{The ``yes" and ``no" parts of the McKinsey algorithm}\label{s:yn}

One of the initial motivations for studying residually finite groups, semigroups and other algebraic structures
was McKinsey's algorithm solving the word problem in finitely presented residually finite algebraic structures.
Even though the algorithm is well known and classical,  surprisingly little is known about its complexity.
In this paper, we fill this gap.

Let $G=\langle X ; R \rangle$ be a residually finite finitely presented algebraic structure of finite type
(signature) $T$ (say, groups, semigroups, rings, etc.)  Let us recall McKinsey's algorithm solving the word problem in $G$
(see \cite{McKinsey}, \cite{Malcev}). The word problem is divided into two parts. Let $F(X)$ be the free
algebraic structure of type $T$ freely generated by $X$. Then we define the ``yes" and ``no" parts of the word problem in $G$ as follows:

$$\mathrm{WP}_{\mathrm{yes}} = \{(w,w') \in F(X) \mid w =_Gw'\} \ \ \text{and} \ \
   \mathrm{WP}_{\mathrm{no}} = \{(w,w') \in F(X) \mid w \neq _G w'\}.$$
To solve the word problem  in $G$ one runs in parallel  two separate algorithms ${\mathcal A}_{\mathrm{yes}}$ and
${\mathcal A}_{\mathrm{no}}$, such that starting with a given pair of elements $w, w' \in F(X)$ ${\mathcal
A}_{\mathrm{yes}}$ stops if and only if $(w,w') \in \mathrm{WP}_{\mathrm{yes}}$ and ${\mathcal A}_{\mathrm{no}}$ stops if
and only if $(w,w') \in \mathrm{WP}_{\mathrm{no}}$.

{\bf The algorithm ${\mathcal A}_{\mathrm{yes}}$}  enumerates one by one all consequences of the defining
relations $R$  and waits until $w=w'$ appears in the list.

{\bf The algorithm ${\mathcal A}_{\mathrm{no}}$}   enumerates all homomorphisms $\phi_1, \phi_2, \ldots,$ of $G$
into finite algebraic structures of type $T$ and waits until  $\phi_i(w) \neq \phi_i(w')$.

Let now $G$ be a finitely presented residually finite group. Although it seems like in general ${\mathcal
A}_{\mathrm{yes}}$ and ${\mathcal A}_{\mathrm{no}}$ are very slow, there were no examples of groups $G$ for which
these algorithms were actually very slow. More precisely, there were no known examples of finitely presented
residually finite groups with very hard ``yes" or ``no" part of the word problem. Indeed, the most ``common"
residually finite groups are linear groups, say, over fields \cite{Malcev}. In that case it is well
known that the ``yes" part can be solved in deterministic polynomial time \cite{LZ77,Waak}. The ``no" part can be
solved by considering factor groups corresponding to ideals of finite index of some polynomial rings, hence also can be shown to be
solvable in deterministic polynomial time. In fact the same can be said about most finitely presented groups (where ``most" means ``with overwhelming probability" in one of several probabilistic models): recent results of Agol \cite{Agol} and Ollivier and Wise \cite{OW} together with the older result of Olshanskii \cite{OlRand} imply that most finitely presented groups  are linear (even over $\Z$).

One of our main results is the following theorem (an immediate corollary of Theorem \ref{t:rhog} below):

\medskip
\noindent {\bf Theorem.} {\em Let $f(n)$ be a recursive function. Then there exists a residually finite finitely presented solvable
group $G$ such that for any finite presentation $\langle X ; R \rangle$ of $G$ the time complexity of both ``yes" and ``no" parts of the word problem are at least
as high as $f(n)$.}

\medskip
We also show that both algorithms ${\mathcal A}_{\mathrm{yes}}$ and ${\mathcal A}_{\mathrm{no}}$ can be very slow
even when both ``yes" and ``no" parts of the word problem are easy.

\begin{remark}\label{r:ele} Note that if we replace ``finitely presented" assumption by ``recursively presented", then residually finite groups are known to be very complicated. As we have mentioned before, recursively presented finitely generated residually finite groups may have undecidable word problem \cite{Meskin, Dyson, Grig}). Even more, recently the second author and B. Khoussainov constructed residually finite Dehn monsters, i.e.,  infinite groups which are  recursively presented, residually finite and algorithmically finite  \cite{KhoussM}. These are groups where the word problem is not only undecidable, but one cannot algorithmically enumerate  an infinite set of pair-wise distinct elements of the group.

Note also that although our groups are not linear they are (elementary Abelian)-by-linear since they are  solvable of class 3 with the  second derived subgroup elementary Abelian.
\end{remark}

\subsection{Quantification of the ``yes" part: the Dehn function}

It was noticed by Madlener and Otto \cite{Madlener-Otto} that the Dehn function of a group measures the complexity of the word problem. They also constructed finitely presented groups with arbitrary large Dehn functions. For residually finite groups, the situation is different.
Nilpotent groups are
examples of residually finite groups with arbitrary high polynomial Dehn function \cite{BMS93}. The
Baumslag-Solitar groups $\la x,y\mid x^y=x^k\ra$, $k\ge 2$, are examples of residually finite (even linear)
groups with exponential Dehn function. No examples of residually finite groups with bigger Dehn functions were
known. This gap is filled by the following

\medskip

\noindent {\bf Theorem \ref{t:rfg}.} {\em For every recursive function $f$, there is a  residually finite finitely presented
solvable of class 3 group $G$ with Dehn function greater than $f$. In addition, one can assume that the word problem in $G$ is at least as hard as any given recursive function or as easy as polynomial time.}

\medskip
As a corollary of Theorem 4.18 we mention the following exotic examples of groups.

\medskip \noindent
{\bf Corollary.} {\em   For every recursive function $f$, there is a  residually finite finitely presented
solvable of class 3 group $G$ with Dehn function greater than $f$ and the word problem decidable in polynomial time.
}
\medskip

\subsection{Quantification of the ``no" part: the depth function}\label{ss:qn}

The function quantifying the algorithm ${\mathcal A}_{\mathrm{yes}}$ is the {\em depth function} introduced by
Bou-Rabee \cite{Bou-Rabee}. Recall that if $G=\la X\ra$ is a finitely generated group or semigroup, the depth
function $\rho_G(n)$ is the smallest function such that every two words $w\ne_G w'$ of length at most $n$ are
separated by a homomorphism to a group (semigroup) $H$ with $|H|\le \rho_G(n)$. That function does not depend on
the choice of finite generating set $X$ (up to the natural equivalence).

It is easy to see that for every finitely generated linear group or semigroup $G$, $\rho_G$ is at most
polynomial. Since finitely generated metabelian groups are subgroups of direct products of linear groups
\cite{W} the depth function of every finitely generated metabelian group is at most polynomial. By the recent
result of Agol \cite{Agol} based on the earlier results of Wise \cite{Wise}, every small cancelation group is a
subgroup of a Right Angled Artin group, hence linear and has polynomial depth function. In fact one can have much
smaller bounds for many linear groups. For example, for the free group $F_2$, $\rho_{F_2}(n)$ is at most
$n^{\frac{2}{3}}$ by a result of Kassabov and Matucci \cite{KM}. There are some finitely presented groups for
which the depth function is unknown and very interesting. For example the ascending HNN extensions of free groups
are known to be residually finite and even virtually residually nilpotent (proved by A. Borisov and the third author \cite{BS05, BS09}) but the only upper bound one can
deduce from the proof is exponential. Although many of these groups have small cancelation presentations
and so covered by the results from \cite{Agol}, there are some groups of this kind for which the depth function is not known.
One of these groups is $\langle x,y,t\mid txt\iv = xy, tyt\iv=yx\rangle$. The fact that it is hyperbolic follows
from Bestvina-Feign combination theorem \cite{Bestvina-Feighn} and was proved by Minasyan (unpublished). If the
depth function of that group is not polynomial, that group would not be linear, disproving a conjecture by Wise
(he conjectured that all hyperbolic ascending HNN extensions of free groups are linear and, moreover, subgroups
of Right Angled Artin groups).

For finitely generated infinitely presented groups (even amenable ones) the situation is much more clear now.
Using the method of Kassabov and Nikolov \cite{Kassabov-Nikolov} and the result of Nikolov and Segal
\cite{Nikolov-Segal} one can construct a finitely generated residually finite group with arbitrary large
recursive depth function.

In this paper, we show that a similar result holds for finitely presented solvable of class 3 groups.

\medskip

\noindent {\bf Theorem \ref{t:rhog}.} {\em For every recursive function $f$, there is a residually finite finitely presented
solvable of class 3 group $G$ with depth function greater than $f$. In addition, one can assume that the word problem in $G$ is at least as hard as the membership problem in a given recursive set of natural numbers $Z$ or as easy as polynomial time.}

\medskip
As a corollary of Theorem 4.20 we mention the following exotic examples of groups.

\medskip \noindent
{\bf Corollary.} {\em   For every recursive function $f$, there is a  residually finite finitely presented
solvable of class 3 group $G$ with depth  function greater than $f$ and the word problem decidable in polynomial time.
}
\medskip

\subsection{Distortion of pro-finitely closed subgroups of finitely presented groups}

Let $G$ be a group generated by a finite set $X$, $H\le G$ be a subgroup generated by a finite set $Y$. Recall
that the distortion function $f_{H,G}(n)$ is defined as the minimal number $f$ such that every element of $H$
represented as a word $w$ of length $\le n$ in the alphabet $X$ can be represented as a word of length $\le f$ in
the alphabet $Y$ \cite{Farb}. It is clear \cite{Farb} that the distortion function $f_{G,H}$ is recursive if and
only if the membership problem in $H$ is decidable.


Recall that $H$ is {\em closed in the pro-finite  topology} of $G$ if $H$ is the intersection of
some subgroups of $G$ of finite index. If $G$ is finitely presented and $H$ is closed in the pro-finite topology of
$G$, then there exists a McKinsey-type algorithm $A(G,H)$ solving the membership problem for $H$ (and thus the
$f_{G,H}$ is recursive). For every word $w$ in the alphabet $X$, the ``yes" part $A_{\mathrm{yes}}(G,H)$ of the
algorithm lists all words in $Y$, rewrites them as words in $X$, and then applies relations of $G$ to check
whether one of these words is equal to $w$.  The ``no" part $A_{\mathrm{no}}(G,H)$ of the algorithm lists all
homomorphisms $\phi$ of $G$ into finite groups and checks whether $\phi(w)\not\in\phi(H)$. As in Section
\ref{s:yn}, one can asks what is the complexity of the ``yes" and ``no" parts of that algorithm, in particular,
and of the membership problem for $H$ in general.

One can also quantify the complexity of the two parts $A_{\mathrm{yes}}(G,H)$ and $A_{\mathrm{no}}(G,H)$. The
``yes" part is quantified by the distortion function $f_{G,H}(n)$ and the ``no" part is quantified by the {\em
relative depth function} $\rho_{G,H}(n)$ which is defined as the minimal number $r$ such that for every word $w$
of length $\le n$ in $X$ which does not represent an element of $H$ there exists a homomorphism $\phi$ from $G$
to a finite group of order $\le r$ such that $\phi(w)\not\in \phi(H)$.

As for the word problem in residually finite finitely presented groups (discussed above), there were no examples
of finitely generated subgroups of finitely presented groups that are closed in the pro-finite topology but have
``arbitrary bad" distortion or ``arbitrary bad" relative depth function.

The Mikhailova's construction  shows that finitely generated subgroups of the
residually finite group $F_2 \times F_2$  (here $F_2$ is a free group of rank 2) could be as distorted as one
pleases. In fact the set of possible distortion functions of subgroups of $F_2\times F_2$ coincides, up to a
natural equivalence, with the set of Dehn functions of finitely presented groups \cite{Olshanskii_Sapir:2001}. By
a result of Baumslag and Roseblade
\cite{Baumslag-Roseblade} subgroups of $F_2\times F_2$ are {\em equalizers} of pairs of homomorphisms $\phi\colon
F_k\to G, \psi\colon F_n\to G$ (where $F_k, F_n$ are subgroups of $F_2$), i.e. the subgroups of the form
$\{(x,y)\in F_k\times F_n\mid \phi(x)=\psi(y)\}$. The equalizer subgroup is finitely generated if and only if $G$
is finitely presented. It is easy to prove (see Lemma \ref{l:89} below) that if $G$ is residually finite, then
the equalizer is closed in the pro-finite topology of $F_2\times F_2$.  Thus we can use the examples of residually
finite finitely presented groups with complicated word problem and complicated depth function to prove the
following

\medskip

\noindent {\bf Theorem \ref{t:distortion}} {\em For every recursive function $f(n)$ there exists a finitely generated subgroup $H \leq F_2\times F_2$ that is
closed in the pro-finite topology of $F_2\times F_2$ and whose distortion function $f_{F_2\times F_2,H}$, the relative depth function, and the
time complexities of both ``yes" and ``no" parts of the membership problem are at least $f(n)$.}

\medskip
There is an analogous (though a bit weaker) result,  for subgroups of a direct product $S_3(X) \times S_3(X)$, where $S_3(X)$ is a free solvable group of class 3 with free generating set  $X$.

\medskip
\noindent {\bf Theorem \ref{t:distS}} {\em For any recursive function $f(n)$ there is a finite set $X$ and a finitely generated  subgroup $H \in S_3(X)
\times S_3(X)$ such that $E$ is closed in the pro-finite topology on $S_3(X) \times S_3(X)$ and
the distortion function, the relative depth function, and the
time complexities of both ``yes" and ``no" parts of the membership problem in $H$ are at least $f(n)$.
}

\subsection{Methods of proof}\label{ss:mop} As we have shown above (see Section \ref{ss:1}) most versions of the Boone-Novikov construction (\cite{Rotman, MT, Cohen, SBR}) do not produce residually finite groups. Instead, we simulate Minsky machines in solvable groups of class 3. A  similar construction is due to the first author \cite{Kh81}. We use a version from the unpublished thesis \cite{Kh90}. As in \cite{Kh81,Kh90}, our group is a split extension of an elementary Abelian group of prime exponent by a metabelian group. Since every metabelian group has easy word problem and is residually finite, we can concentrate only on the elementary Abelian subgroup which is spanned, basically, by the configurations of the Minsky machine encoded in a certain way. That encoding is a very important feature of the construction. It helps us avoid the problem from Section \ref{ss:1} because the words $w(u)$ corresponding to the input configurations $u$ of the Minsky machines are not obtained by inserting copies of $u$ in $w(\emptyset)$.

Of course that construction also often leads to non-residually finite groups. But it turned out that the difficulty can be overcome by modifying the Minsky machine first. In this paper, we use the fact that every Turing machine and every Minsky machine with decidable halting problem is equivalent to a {\em universally halting} and even {\em sym-universally halting} machine (these machines can be characterized as the machines whose transition graphs are vertex disjoint unions of finite trees).

\subsection{Structure of the paper} The paper is organized as follows. Section \ref{s:1} contains preliminary
results about Turing and Minsky machines that are needed further. We show that one can modify any Turing or
Minsky machine that recognizes a recursive set into a machine that halts on every configuration. In fact we can
even assume that the symmetrized machine always halts (we call such machines sym-universally halting).

In Section \ref{s:2}, we simulate sym-universally halting Minsky machines in residually finite finitely presented
semigroups and prove the analogs of the above theorems for semigroups.

In Section \ref{s:4} we simulate Minsky machines in solvable groups and construct complicated residually finite finitely presented groups.

Sections \ref{s:dis} and \ref{s:6} contain applications of the main theorems.  In Section \ref{s:dis} we prove,
in particular, Theorem \ref{t:distortion}. In Section \ref{s:6}, we strengthen the well-known result of
Slobodskoi about undecidability of the universal theory of finite groups. We show, in particular, that the
universal theory of any set of finite groups that contains all finite solvable groups of class 3 is undecidable.

{\bf Acknowledgement.} The authors are grateful to Jean-Camille Birget and Friedrich Otto for pointing to the references \cite{MDavis}, to Ben Steinberg for pointing to the reference \cite{LZ77},to Rostislav Grigorchuk for pointing to the references \cite{Dyson,Grig} and to Tim Riley for pointing to the references \cite{Ger,GR}. We are also grateful to Markus Lohrey and Ralph Strebel for their comments.

\section{Turing machines and Minsky machines}\label{s:1}

\subsection{Turing machines} Let us give a definition of a Turing machine. A Turing machine $M$ with $K$ tapes
consists of hardware (the tape alphabet $A=\sqcup_{i=1}^k A_i$, and the state alphabet $Q=\sqcup_{i=1}^K
Q_i$\footnote{$\sqcup$ denotes disjoint union}) and program $P$ (the list of commands, defined below). A {\em
configuration} of a Turing machine $M$ is a word $$\alpha_1 u_1q_1v_1\omega_1\hbox{  } \alpha_2u_2q_2v_2\omega_2\hbox{  }
\ldots\hbox{  } \alpha_Ku_Kq_Kv_K\omega_K$$ (we included spaces to make the word more readable) where $u_i, v_i$ are
words in $A_i$, $q_i\in Q_i$ and $\alpha_i,\omega_i$ are special symbols (not from $A\cup Q$).

A command simultaneously replaces subwords $a_iq_ib_i$ by words $a_i'q_i'b_i'$ where $a_i,a_i',$ are either
letters from $A_i\cup\{\alpha_i\}$ or empty, $b_i,b_i'$ are either letters from $A_i\cup \{\omega_i\}$ or empty.
A command cannot insert or erase $\alpha_i$ or $\omega_i$, so if, say, $a_i=\alpha_i$, then $a_i'=\alpha_i$. Note
that with every command $\theta$ one can consider the {\em inverse} command $\theta\iv$ which undoes what
$\theta$ does.

A {\em computation} of $M$ is a sequence of configurations and commands from $P$:

$$w_1 \stackrel{\theta_1}{\longrightarrow} w_2 \stackrel{\theta_2}{\longrightarrow} \ldots
\stackrel{\theta_l}{\longrightarrow} w_{l+1}.$$

Here $l$ is called the {\em length} of the computation. We choose {\em stop states} $q_i^0$ in each $Q_i$, then
we can call a configuration $w$ {\em accepted} if there exists a computation starting with $w$ and ending with a
configuration where all state symbols are $q_i^0$ and all tapes are empty. Also we choose {\em start states}
$q_i^1$ in each $Q_i$. Then an {\em input} configuration corresponding to a word $u$ over $A_1$ is a
configuration $\inp(u)$ of the  form $$\alpha_1uq_1^1\omega_1\hbox{  } \alpha_2q_2^1\omega_2\hbox{  } \ldots\hbox{  } \alpha_Kq_K^1\omega_K.$$
We say that a word $u$ over $A_1$ is accepted by $M$ if the configuration $\inp(u)$ is accepted. The set of all
words accepted by $M$ is called the {\em language accepted by} $M$.

The {\em time function} $T_M(n)$ of $M$ is the minimal function such that every accepted word of length $\le n$
has an accepting computation of length $\le T_M(n)$. The {\em space function} $S_M(n)$ of $M$ is the minimal
function such that every accepted word of length $\le n$ has an accepting computation where every configuration
has length $\le S_M(n)$.

A Turing machine $M$ is called {\em deterministic} if for every configuration, there exists at most one command
from the program $P$ that applies to this configuration.

In this paper, we shall consider several types of machines. A machine $M$ in general has an alphabet and a set of
words in that alphabet called {\em configurations}. It also has a finite set of commands. Each command is a
partial injective transformation of the set of configurations. A computation is a sequence

$$w_1 \stackrel{\theta_1}{\longrightarrow} w_2 \stackrel{\theta_2}{\longrightarrow} \ldots
\stackrel{\theta_l}{\longrightarrow} w_{l+1}.$$
where $w_j$ are configurations, $\theta_1,...,\theta_n$ are commands and $\theta_i(w_i)=w_{i+1}$ for every
$i=1,...,n$. A machine is called {\em deterministic} if the domains of its commands are disjoint. A machine
usually has a distinguished {\em stop} configuration, and a set $I=I(M)$ of {\em input} configurations. A
configuration is called {\em accepted} by $M$ if there exists a computation connecting that configuration with
the stop configuration. The machine $\Sym(M)$ is defined in the natural way (add the inverses of all commands of
$M$). Two configurations $w$, $w'$ are called {\em equivalent}, written $w\equiv_M w'$, if there exists a
computation of $\Sym(M)$ connecting these configurations. Clearly, $\equiv_M$ is an equivalence relation.

The following general lemma is easy but useful.

\begin{lemma}\label{l:genmachine}\label{l:sym} Suppose that $M$ is deterministic. Then two configurations $w,w'$
of $M$ are equivalent if and only if there exists two computations of $M$ connecting $w,w'$ with the same
configuration $w''$ of $M$.
\end{lemma}

\proof Indeed, since $M$ is deterministic, in any computation of $\Sym(M)$ where no command is followed by its
inverse inverses of command of $M$ cannot be followed by commands of $M$. Thus the computation is a concatenation
of two (possibly empty) parts: the first part uses only commands of $M$, the second part uses only inverses of
commands of $M$.
\endproof

We say that a set $X$ of natural numbers is {\em enumerated} by a machine $M$ if there exists a recursive
encoding $\mu$ of natural numbers by input configurations of $M$ such that a number $u$ belongs to $X$ if and
only if $\mu(u)$ is accepted by $M$. The set $X$ is {\em recognized} by $M$ if $M$ enumerates $X$ and for every
input configuration every computation starting with that configuration eventually halts (arrives to a
configuration to which no command of $M$ is applicable).

We say that machine $M'$ {\em polynomially reduces} to a machine $M$ if there exists an polynomial time algorithm
$A$ checking equivalence of configurations of $M'$ which uses an oracle checking equivalence of configurations of
$M$ such that

\begin{itemize}
\item Any computation of $A$ verifying equivalence of configurations $c,c'$ of $M'$ involves at most
    polynomial (in terms of $|c|+|c'|$ ) number of uses of the oracle,
\item and every time the sizes of the configurations of $M$ whose equivalence the oracle should check are
    polynomially bounded in terms of $|c|+|c'|$.
\end{itemize}

We say that $M$ and $M'$ are {\em polynomially equivalent} if there are polynomial reductions of $M$ to $M'$ and
vice versa.

\subsection{Universally halting Turing machines} A (not necessarily deterministic) machine $M$ is called
{\em universally halting} if for every configuration $w$ of $M$ there exist only finitely many computations of
$M$ starting with $w$ without repeated configurations.

We call a deterministic machine $M$ sym-universally halting if $\Sym(M)$ halts if it starts with any non-accepted
configuration.

\begin{theorem}[See, for example, \cite{MDavis}]\label{??} For every recursive set $X$ of natural numbers, that is accepted by a
deterministic Turing machine $M$ there exists a
universally halting deterministic Turing machine $M'$ with one tape accepting $X$ and polynomially equivalent to
$M$.
\end{theorem}

\begin{lemma}\label{l1}   Let $M$ be  a deterministic sym-universally halting Turing machine. Then there exists a
one-tape deterministic sym-universally halting Turing machine $M'$ recognizing the same language as $M$. The
machine $M'$ is polynomially equivalent to $M$.
\end{lemma}

\proof The proof is by inspection of the proof from \cite{CJ1}.\endproof

\begin{theorem} \label{t:sym} For every recursive set of natural numbers  $X$ there exists a sym-universally
halting Turing
machine $M''$ with one tape that recognizes $X$. The machine $M''$ satisfies the following conditions.

(a) For every configuration $c$ of $M''$ either $c$ is equivalent to an input configuration or every computation
of $\Sym(M'')$  starting with $c$ has length at most $O(|c|)$.

(b) If $c,c'$ are two distinct input configurations of $M''$ such that $c\equiv_{M''} c'$. Then either $c=c'$ or
both $c,c'$ are accepted by $M''$.

(c) If $M$ is any Turing machine recognizing $X$ then we can assume that $M''$ polynomially reduces to $M$.
\end{theorem}

\proof Let $M$ be a deterministic universally halting Turing machine with $K$ tapes recognizing $L$. Consider a
new Turing machine $M'$ constructed as follows. $M'$ has one more tape than $M'$, called the {\em history} tape.
Its alphabet $A'$ is in one-to-one correspondence with the set of commands $P$ of $M$: $A'=\{[\theta],\theta\in
P\}$. Its state alphabet consists of two letters $q_{K+1}^0$ and $q_{K+1}^1$. With every command $\theta$ of $M$
we associate the command $\theta'$ of $M'$. It does what $\theta$ would do on the first $k$ tapes of $M'$ and
inserts $[\theta]$ on the history tape of $M$. After the first $K$ tapes of $M'$ form an accept configuration,
the machine erases the history tape and stops (turns $q_{K+1}^1$ into the stop state $q_{K+1}^0$). Let $P'$ be
the program of $M'$. Now modify $M'$ further to obtain a new Turing machine $M''$. The program $P''$ of $M'$
contains a copy $\tilde P$ (the set of the {\em main commands}) of $P'$ and some new, auxiliary, commands. After
each main command of $\tilde P$, $M''$ executes the history written on the history tape backward, without erasing
the history tape: it just scans the tape from left to right, reading the symbols written there one by one and
executing on the first $K$ tapes the inverses of the commands written on the history tape. The commands that do
that will be called {\em auxiliary}. If at the end of the scanning the history tape, it reaches an input
configuration, $M''$ executes on the first $K$ tapes the history written on the history tape in the natural order
(scanning the history tape from right to left). After that $M''$ is ready to execute the next main command. We do
not give precise definition of the program of $M''$ because it is obvious on the one hand and long on the other
hand. Clearly, the state alphabet of $M''$ must be bigger than the state alphabet of $M'$. The machine $M''$ is
deterministic, universally halting,  and recognizes the same language $L$.

Let us prove properties (a) and (b) of the theorem. Since $M''$ is deterministic, every reduced (i.e. without
mutually inverse consecutive commands) computation $\Theta$ of $\Sym(M'')$ is of the form $\Theta_1\Theta_2\iv$
for some computations $\Theta_1, \Theta_2$ of $M''$ (because a command of $(M'')\iv$ cannot be followed by its
inverse).

Let us show that $\Sym(M'')$ halts when it starts with any non-accepted configuration (and then apply Lemma
\ref{l1}). Let $w$ be a configuration of $M''$ that is not accepted by $M''$. Since $M''$ is deterministic, every
computation of $\Sym(M'')$ starting at $w$ is a concatenation of a computation of $M''$ followed by a computation
of $(M'')\iv$ (i.e. the machine $M''$ where every command is replaced by its inverse). Since $M$ is universally
halting, there are only finitely many computations of $M''$ starting with $w$. Thus we only need to show that
there are finitely many computations of $(M'')\iv$ starting with $w$, or, equivalently, that there are only
finitely many computations of $M''$ ending with $w$. Suppose that there are infinitely many computations of $M''$
ending with $w$. Then, by definition of $M''$ there must exist infinitely many input configuration $\inp(u)$ of
$M''$ for which there exists a computation of $M''$ starting with $\inp(u)$ and ending at $w$. But that is
impossible because such $\inp(u)$ is unique and is obtained by applying the inverse of the history written on the
history tape of $M''$ to $w$.

(c) The fact that $M''$ polynomially reduces to $M$ is proved as follows. Consider two configurations $w,w'$ of
$M''$. If $w$ is not equivalent to an input configuration, then by (b) we need to check only whether $w'$ is one
of $O(|w|)$ words that belong to the longest computation of $\Sym(M'')$ containing $w$. That can be done in
polynomial time without using the oracle checking equivalence of configurations of $M$. Suppose that both $w$ and
$w'$ are equivalent to input configurations $u, u'$ of $M''$. Then we can find $u,u'$ in polynomial time and
their lengths at at most $O(|w|+|w'|)$. If $u\ne u'$ and either $u$ or $u'$ is not accepted, then by (b) $w$ is
not equivalent to $w$. If $u=u'$, then $w$ is equivalent to $u$. Thus $w$ is equivalent to $w'$ if and only if
$u$ and $v$ are accepted. To check that $u$ is accepted, we need to remove letters corresponding to the extra
tape from $u$ producing a configuration $u_1$ of $M$ and check whether $u_1$ is accepted, i.e. whether $u_1$ is
equivalent to the stop word of $M$. This can be done by asking the oracle once. Thus to check whether $w$ and
$w'$ are equivalent we only need polynomial time and asking the oracle about equivalence of two pairs of
configurations, the lengths of which are bounded by $|w|+|w'|$. Thus $M''$ polynomially reduces to $M$.
\endproof

\subsection{Minsky machines}

The hardware of a $K$-glass Minsky machine, $K\ge 2$, consists of $K$ glasses containing coins. We assume that
these glasses are of infinite height. The machine can add a coin to a glass, and remove a coin from a glass
(provided the glass is not empty). The commands of a Minsky machine are numbered. So a configuration of a
$K$-glass Minsky machine is a $K+1$-tuple $(i;\epsilon_1,\ldots,\epsilon_K)$ where $i$ is the number of command
that is to be executed, $\epsilon_j$ is the number of coins in the glass $\# j$.

More precisely, a {\em command} has one of the following forms:

\begin{itemize}
\item Put a coin in each of the glasses $\#\# n_1,...,n_l$ and go to command \# $j$. We shall encode this
    command as
$$ i;\to \Add(n_1,...,n_l); j$$ where $i$ is the number of the command;
\item If the glasses $\#\# n_1,...,n_l$ are not empty then take a coin from each of these glasses
and go to instruction \# $j$. This command is encoded as $$ i; \epsilon_{n_1}>0,...,\epsilon_{n_l}>0\to
Sub(n_1,...,n_l); j; $$
\item  If glasses $\#\# n_1,...,n_l$ are empty, then  go to instruction \# $j$. This command is encoded as
$$ i; \epsilon_{n_1}=0,...,\epsilon_{n_l}=0\to j; $$

\item{Stop.} This command is encoded as
$i;\to 0;$
\end{itemize}

\begin{remark} This defines deterministic Minsky machines. We will also need non-deterministic Minsky machines.
Those
will have two or more commands with the same number.
\end{remark}

\begin{theorem}\label{t:MM}
Let $X$ be a  recursively enumerable set of natural numbers. Then the following holds:
\begin{itemize}
\item[(a)] there exists a 2-glass deterministic Minsky machine $\mathit{MM}_2$ which recognizes
$L$ in the following sense: $\mathit{MM}_2$   begins its work in configuration $(1;2^m,0)$ and stops in configuration
$(0;0,0)$ if  and only if $m\in X$,  and it  works forever  if $m\not\in X$.
\item[(b)] There exists a $3$-glass Minsky machine $\mathit{MM}_3$ which when started on a  configuration
    $(1;m,0,\ldots ,0)$ stops in the configuration $(q_0,0,0,\ldots ,0)$ provided $m\in X$,
and works forever otherwise.

\item[(c)] We can also assume that every computation of $\mathit{MM}_2$ or $\mathit{MM}_3$ starting with a configuration $c$
    empties each glass after at most $O(|c|)$ steps.

\item[(d)] If $X$ is recursive, then the machine $\mathit{MM}_3$ above can be chosen to be sym-universally
halting.

\item[(e)] If $M$ is a deterministic Turing machine recognizing $X$, then we can assume that $\mathit{MM}_2$ (resp.
    $\mathit{MM}_3$) polynomially reduces to $M$ where the numbers written on the tapes of $M$ are measured as
    represented in unary (that is the size of a number $n$ is set as $n$ and not $\log_2 n$).
\end{itemize}
\end{theorem}

\proof The proof of the 2-glass part can be found in \cite{malcev}. Let us prove the $3$-glass part of the
theorem. Let $M$ be a one tape deterministic Turing machine $M$ recognizing $L$. Without loss of generality we
can assume that the tape alphabet of $M$ is $\{1,2\}$. For  every configuration $\alpha u q_i v\omega$ of $M$ we
can view $u$ and $v$ as numbers written in 3-ary, where $v$ is read from right to left. Let us denote these
numbers by $l(u), r(v)v$. For example the numbers corresponding to the configuration $\alpha 121221 q_5 1222\omega$ are
$ l(u)=121221_3$ and $r(v)=2221_3$, both written as 3-ary numbers. Thus with the configuration $\alpha u q_i v
\omega$, we associate the following configuration of a 3-tape Minsky machine $(i;l(u),r(v),0)$. Now every
command of a Turing machine can be simulated by a series of commands of the Minsky machine. For example, the
command $\Theta$ of the form $1 q_i 2\to q_j 1$ is interpreted by a sequence $M(\theta)$ of commands of $\mathit{MM}_3$ as
follows. The commands of $M(\theta)$ will be numbered $i.1$ through $i.l$ for some $l$. The commands from
$M(\theta)$ should replace $l(u)$ coins in the first glass by $\lfloor l(u)/3\rfloor$ coins provided $l(u)\equiv
1 \mod 3$ and replace $r(v)$ coins in the second glass by $3\lfloor r(u)/3\rfloor +1$ coins provided $r(u)\equiv
2 \mod 3$. The first part is done as follows.  Decrease the number of coins in the first glass by 3
simultaneously increasing the number on the third glass by $1$. Do that until the number of coins in the first
glass is less than $3$. If that remaining number is 1, then subtract 1 coin from the first glass, and then keep
adding 1 coin to the first glass, removing 1 coin from the third glass until the third glass is empty. If the
remainder is not 1, then keep removing one coin from the third glass while adding 3 coins to the first glass -
until the third glass is empty (i.e. return to the original configuration because the command $\theta$ is not
applicable). The second part is done in a similar manner by using the second and third glasses of the Minsky
machines. Other commands of the Turing machine are treated in the same manner. Let $\mathit{MM}_3$ be the resulting 3-tape
Minsky machine. It is easy to see that if the Turing machine $M$ is sym-universally halting, then the Minsky
machine $\mathit{MM}_3$ is sym-universally halting. This gives properties (a),(b) and (d) of the theorem.

To ensure Property (c), we can do the following. Note that after every series of commands $M(\theta)$ the
configuration of $\mathit{MM}_2$ or $\mathit{MM}_3$  has (at least) one empty glass. After the series of commands $M(\theta)$ of
$\mathit{MM}_2$ or $\mathit{MM}_3$ corresponding to a command $\theta$ of $M$ is executed, that glass is again empty. So before
$\mathit{MM}_2$ or $\mathit{MM}_3$ execute the next series $M(\theta')$ we force it to move all coins from each of the non-empty
glasses to the empty one and back. In the process, it will empty each glass at least once. Clearly, this
modification increases the length of computation by an amount proportional to the length of configuration of
$\mathit{MM}_2$ or $\mathit{MM}_3$.

Finally property (e) is obtained as follows. Suppose that $w,w'$ are two configurations of $\mathit{MM}_3$ (for $\mathit{MM}_2$ the
proof is similar). By construction (see \cite{malcev}) in at most $O(|w|)$ steps of $\mathit{MM}_3$ either $w$ turns into
a configuration corresponding to a configuration of the Turing machine $M$ or $\mathit{MM}_3$ halts. In the latter case,
we check whether $w$ is equivalent to $w'$ in $O(|w|)$ steps. So we can assume that both $w$ and $w'$ are
equivalent to configurations corresponding to configurations $u,u'$ of the Turing machine $M$ whose lengths are
$O(|w|+|w'|)$. Now $w$ is equivalent to $w'$ if and only if $u$ and $u'$ are equivalent configurations of $M$.
Thus we need to use the oracle once.
\endproof

\section{Simulation of Minsky machines by semigroups} \label{s:2}

\subsection{The construction}

Here we will show how to simulate a Minsky machine by a semigroup. All applications of Minsky machines are based on
the following idea.

First, with every configuration $\psi$ one associates a word (term) $w(\psi)$.

Then with every command $\kappa$ of the Minsky machine $M$ one associates a finite set of defining relations
$R_\kappa$.  The algebraic structure $A(M)$ will be defined by the relations from the union $R$ of all $R_\kappa$ (which is
finite since we have only a finite number of commands) and usually some other relations $Q$ which are in a sense
``independent" of $R$. We need $Q$, for example, to make sure $A(M)$ satisfies a particular identity.

We say that the algebra $A(M)$  {\em simulates} $M$ if the following  holds for arbitrary configurations $\psi_1,
\psi_2$ of $M$:

\begin{equation}
\psi_1\equiv_M\psi_2 \hbox{ if and only if } w(\psi_1)=w(\psi_2) \hbox{ in } A(M).
\label{eq:interpretation}
\end{equation}

Usually, in order to prove the property (\ref{eq:interpretation}) one has to prove the following two lemmas.

\begin{lemma} \label{forward} If  a configuration $c'$  can be obtained from a  configuration $c$  by a command
$\kappa$ of $M$ then the word
$w(c')$ can be obtained from the word $w(c)$ by applying  defining relations of $A(M)$  from the set $R_\kappa$.
\end{lemma}

\begin{lemma} If  a word $w(c')$ can be obtained from a word $w(c)$ by applying  the defining relations of $A(M)$
then  $c\equiv_M c'$.
\label{llemma2}
\end{lemma}

It is easy to see that Lemmas \ref{forward} and \ref{llemma2} imply property (\ref{eq:interpretation}).

There is an easy way to interpret Minsky machines in a semigroup $S(M)$. Let $M$ be a Minsky machine with $K$
glasses and commands $\#\# 1,2,...,N, 0$. Then  $S(M)$ is generated by the elements $q_0,\ldots,q_N$ and
$\{a_i,A_i, i=1,...,K\}$. The set of defining relations of $S(M)$ consists of all commutativity relations

\begin{equation}\label{ec}
a_ia_j=a_ja_i, a_iA_j=A_ja_i, A_iA_j=A_jA_i, i\ne j,
\end{equation}
which we shall call {\em commutativity relations}, the {\em stop relation}
\begin{equation}\label{es}
q_0=0 \hbox{ (i.e. }q_0x=xq_0=q_0 \hbox{ for every generator } x\hbox{)},
\end{equation}
all relations of the form $xy=0$ where $xy$ is a two-letter word which is {\it not} a subword of a word of the
form $q_ia_1^{\epsilon_1}...a_K^{\epsilon_K}A_1...A_K$ modulo the commutativity relations (\ref{ec}), (for
example $q_iq_j=A_ia_i=a_iq_j=A_iq_j=0$), which we shall call {\em $0$-relations}, and relations associated with
commands of $M$ according to the following table,

\begin{equation}
\label{eq:s2}
\begin{array}{|l|l|}
\hline
\hbox{Command of } M& \hbox{Relation of } S(M) \\
\hline
&\\ i\rightarrow \Add(n_1,...,n_m);j & q_i=q_ja_{n_1}...a_{n_m}\\
\hline
&\\ i,\epsilon_{n_1}>0,...,\epsilon_{n_m}>0\to \Sub(n_1,...,n_m);j & q_ia_{n_1}\ldots a_{n_m}=q_j\\
\hline
 &\\
i,\epsilon_{n_1}=0,...,\epsilon_{n_m}=0\to j  & q_iA_{n_1}\ldots A_{n_m}=q_jA_{n_1}\ldots A_{n_m}\\
\hline
\end{array}
\end{equation}
These will be called the Minsky relations.

The  words in $S(M)$ corresponding to configurations of $M$ are the following: $$w(i;\epsilon_1,...,\epsilon_K)=
q_ia_1^{\epsilon_1}...a_K^{\epsilon_K}A_1...A_K.$$

The proof that Lemmas \ref{forward} and \ref{llemma2} hold in $S(M)$ follows easily from Lemma \ref{l:sym}, see
\cite{Sapir, KSs}.

\begin{lemma}\label{l:l9} Suppose that a word $W$ is not 0 in $S(M)$, is a subword of
$w(i;\epsilon_1,...,\epsilon_K)$ (up to the commutativity relations (\ref{ec})), and does not contain either
$q_i$ or one of the $A_j$. Then there are at most $O(|W|)$ different (up to the commutativity relations) words
that are equal to $W$ in $S(M)$. All these words are subwords of words of the form
$w(i';\epsilon_1,...,\epsilon_K)$ such that the configurations $(i;\epsilon_1,...,\epsilon_K)$ and
$(i',\epsilon_1',...,\epsilon_K')$ of $M$ are equivalent.
\end{lemma}
\proof Since $W\ne 0$ in $S(M)$, the stop relations do not apply to $W$ or to any word that is equal to $W$ in
$S(M)$. If $W$ does not contain $q_i$, then the only relations that apply to $W$ are the commutativity relations,
so the only words that are equal to $W$ in $S(M)$ are the words obtained from $W$ by the use of commutativity
relations.

Suppose that $W$ contains $q_i$ but does not contain one of the $A_j$.

Without loss of generality, we can assume that $W$ contains every letter from $w(i;\epsilon_1,\ldots,\epsilon_K)$
except some of the $A_j$'s.

Every application of the Minsky relation to $W$ corresponds to a command of the Minsky machine, applied to the
configuration $c=(i;\epsilon_1,...,\epsilon_K)$. Let $c=c_1\to c_2\to...$ be any computation of $\Sym(M)$
starting with $c$. Then the sequence of commands of $M$ applied in that computation has the form
$\theta_1\ldots\theta_n\theta_{n+1}\iv\ldots\theta_{k}\iv$ where $\theta_s$ are commands of $M$ (by Lemma
\ref{l:sym}). If this sequence can be applied to $W$, then this computation never checks whether glass $\# j$ is
empty. By Property (c) of Theorem \ref{t:MM}, both $n$ and $k$ must be at most $O(|W|)$. This implies the
statement of the lemma.
\endproof

\subsection{Residually finite finitely presented semigroups}

\begin{lemma}\label{l:del} Every non-zero element of $S(M)$ is represented by a subword of a word of the form
$w(i;\epsilon_1,...,\epsilon_K)$.
\end{lemma}

\proof This follows from the commutativity relations and 0-relations.
\endproof

\begin{lemma}\label{l:fd} Suppose that the Minsky machine $M$ is sym-universally halting. Then

(a) Every non-zero element $z$ of $S(M)$ has finitely many divisors, i.e. elements $y$ such that $z=pyq$ for some
$p,q\in S(M)\cup\{1\}$.

(b) For every configuration $\psi$ of $M$ the word $w(\psi)$ is equal to 0 in $S(M)$ if and only if $\psi$ is
accepted by $M$.
\end{lemma}

\proof (a) This follows from Lemmas \ref{l:sym}, \ref{forward} and \ref{llemma2} for words containing a
$q$-letter and all letters $A_j$, and from Lemma \ref{l:l9} in all other cases since word in the generators of
$S(M)$ that is non-zero in $S(M)$ is a subword of one of the words corresponding to configurations of the Minsky
machine $M$ by Lemma \ref{l:del}.

(b) This follows from Lemmas \ref{forward} and \ref{llemma2}.
\endproof

Lemma \ref{l:fd} immediately implies

\begin{lemma}\label{l:fd1} For every $R>0$ let $V_R$ be the set of all elements of $S(M)$ that do not divide in
$S(M)$ non-zero elements represented by words of the form $q_ja_1^{\epsilon_1}\ldots a_K^{\epsilon_K}A_1\ldots
A_K$ with $\epsilon_j\le R$. Then $V_R$ is an ideal of $S(M)$ with a finite complement. If $M$ is sym-universally
halting, then the intersection of all $V_R, R>0,$ is $\{0\}$.
\end{lemma}

\begin{theorem}\label{t:rc} For every recursive set of natural numbers $Z$ there exists a residually finite
semigroup $S$ whose word problem is at least as hard as the membership problem in $Z$. The Dehn function of $S$
is equivalent to the time function of a 3-glass Minsky machine recognizing $Z$.
\end{theorem}

\proof By Theorem \ref{t:sym}, there exists a sym-universally halting Turing machine that recognizes $Z$. By
Theorem \ref{t:MM} there exists a sym-universally halting Minsky machine $M$ recognizing $Z$. By Lemma
\ref{l:fd}, the problem of recognizing equality to 0 in $S(M)$ is at least as hard as the membership problem in
$Z$. By Lemma \ref{l:fd1}, $S(M)$ is residually finite.
\endproof

\subsection{Residually finite semigroups with large depth function}

Recall the definition of the depth function $\rho$: for every finitely generated residually finite universal
algebra $A$ and every number $n$, $\rho_A(n)$ is defined as the smallest number such that for every two different
elements $z,z'$ in $A$ of length $\le n$ there exists a homomorphism $\phi$ from $A$ onto a finite algebra $B$ of
cardinality at most $\rho_A(n)$ such that $\phi(z)\ne \phi(n')$.

The following lemma is well known \cite{Golubov}

\begin{lemma}\label{l:s} Suppose that every non-zero element of a semigroup $S$ with $0$ has finitely many
divisors. Then $S$ is residually finite.
\end{lemma}
\proof Indeed, the set of all non-divisors of a non-zero element is an ideal with finite quotient. The
intersection of all these ideals is $\{0\}$.
\endproof

\begin{theorem}\label{t:rho} For every recursive function $f$ there exists a finitely presented residually
finite semigroup $S$ such that $\rho_S(n)>f(n)$ for all $n$. In addition, we can assume that the word problem in
$S$ is as hard as the membership problem for any prescribed recursive set of natural numbers.
\end{theorem}

\proof Let $M$ be a sym-universally halting Minsky machine with $K$ glasses and $N+1$ commands numbered
$0,\ldots,N$. Consider the following new, non-deterministic Minsky
machine $M_n$. Its hardware consists of the $K$ glasses of $M$ plus two more glasses.  In every command of $M$ we add
the instruction to add a coin to glass $K+1$ provided glass $K+2$ is empty. Also for every $i=0,...,N$ we add two
new commands number $i$

\begin{equation}\label{c:a}
i; \Add(K+1,K+2)\to i
\end{equation}
and

\begin{equation}\label{c:s}
i; \epsilon_{K+1}=0, \epsilon_{K+2}=0, \to 0
\end{equation}

Thus there will be three commands for each $i=1,\ldots,N$: one from $M$ and the two new ones. The new command
(\ref{c:a}) allows us to add, at any step of the computation, equal (but arbitrary) number of coins in glasses
$K+1$ and $K+2$, and if both glasses $K+1$ and $K+2$ are empty, the computation can stop. But we can execute a
command of $M$ only when the glass $K+2$ is empty, so a new command cannot be followed by a command of $M$.

Let us say that the commands coming from  $M$ have weight 1 and new commands (\ref{c:a}), (\ref{c:s}) have weight
0. The weight of a computation is then the sum of the weights of all commands used in the computation. We also
define the weight of a configuration as the number of coins in the first $K+1$ glasses minus the number of coins
in glass $K+2$. Every computation $C$ of $\Sym(M_n)$ projects onto a computation $\pi(C)$ of $\Sym(M)$: we simply
forget the extra two glasses and the new commands. The weight of $C$ is equal to the length of $\pi(C)$. The
numbers of coins used in $C$ and $\pi(C)$ in the first $K$ glasses are the same, the number of coins in glass
$K+1$ in the last configuration $W$ of $C$  minus the number of coins in glass $K+2$ of $W$ is equal to the
weight of $C$.

Also any computation $C$ of $M$ lifts to (possibly infinitely many) computations $C_n$ of $M_n$, the weight of
each $C_n$ is the same as the length of $C$, and the number of coins used in the first $K$ glasses is the same.

Note that Lemma \ref{l:sym} still holds for $M_n$ even though $M_n$ is non-deterministic. It can be easily
established by using the projection $\pi$.

This implies that if $M$ is sym-universally halting, then for every configuration $W$ of $M_n$ the weights of all
computations $C$ without repeated configurations of $\Sym(M_n)$  are bounded, the number of coins in the first
$K$ glasses of $M_n$ used during any of these computations is bounded, and the weights of configurations
appearing in these computations are bounded.

Consider the semigroup $S(M_n)$. Every non-zero element $w$ in $S(M_n)$ is represented by a word of the form
$u(w)v(w)$ where $$u(w)=q_i^{\alpha_0}a_1^{l_1}\ldots a_K^{l_K}A_1^{\alpha_1}\ldots A_K^{\alpha_K},
v(w)=a_{K+1}^{l_{K+1}}a_{K+2}^{l_{K+2}}A_{K+1}^{\alpha_{K+1}}A_{K+2}^{\alpha_{K+2}}$$ where $\alpha_j\in \{0,1\},
l_j\ge 0$. Note that if two non-zero words $w$, $w'$ are equal in $S(M_n)$, then $u(w)$ and $u(w')$ are equal in
$S(M)$.

We claim that $S(M_n)$ is residually finite. Indeed, consider two words $w_1, w_2$ in the generators of $S(M_n)$
which are not equal in $S(M_n)$, $w_2$ does not divide $w_1$ in $S(M_n)$ (clearly $w_1$ and $w_2$ cannot divide
each other without being equal in $S(M_n)$).

Suppose first that $w_1$ does not contain a $q$-letter. Then consider the ideal $Q$ of $S(M_n)$ generated by all
$q$-letters. The inequality $w_1\ne w_2$ survives in the Rees factor-semigroup $S(M_n)/Q$. But in $S(M_n)/Q$
every element has finitely many divisors, hence $S(M_n)/C$ is residually finite by Lemma \ref{l:s}, and so we can
separate $w_1$ and $w_2$ by a homomorphism onto a finite semigroup.

Thus we can assume that $w_1$ starts with a $q$-letter $q_i$. Suppose that $u(w_1)\ne u(w_2)$ in $S(M)$. Adding
the relation $a_{K+1}^2=a_{K+1}, a_{K+2}^2=a_{K+2}^2$ to $S(M_n)$ we then obtain a new semigroup $\bar S(M_n)$
and a homomorphism $\phi\colon S(M_n)\to \bar S(M_n)$ which separates $w_1$ and $w_2$. In the semigroup $\bar
S(M_n)$, every non-zero element has finitely many divisors since it is true for $S(M)$ and the number of
different elements of the form $v(w)$ is finite. Hence $\bar S(M_n)$ is residually finite by Lemma \ref{l:s}.

Thus we can assume that $u(w_1)=u(w_2)$. Let
$$v(w_1)=a_{K+1}^{m_1}a_{K+2}^{m_2}A_{K+1}^{\beta_{K+1}}A_{K+2}^{\beta_{K+2}},\hbox{      }
D=\max\{|l_{K+1}-l_{K+2}|+1,|m_{K+1}-m_{K+2}|+1\}.$$ Let us add the relations $a_{K+1}^D=a_{K+1}^{2D},
a_{K+2}^D=a_{K+2}^{2D}$ to $S(M_n)$. Let $\tilde S(M_n)$ be the resulting semigroup, and $\psi\colon S(M_n)\to
\tilde S(M_n)$ be the corresponding homomorphism. Then it is easy to see that $\psi(w_1)\ne \psi(w_2)$. Since in
$\tilde S(M_n)$, every element has finite number of divisors (the same argument as for $\bar S(M_n)$), we can
again use Lemma \ref{l:s}.

The function $\rho(n)$ for the semigroup $S(M_n)$ is at least as large as the following function $\Psi(n)$
associated with the machine $M$: $\Psi(n)$ is the smallest number such that for every non-accepted input
configuration of $M$ of length $\le n$, the machine $M$ halts after at most $\Psi(n)$ steps (i.e. the {\em
co-time function} of $M$). Indeed let $c$ be an input configuration of length at most $n$ such that $M$ halts
after exactly $\Psi(n)$ steps starting at $c$. Suppose that the word $w(c)$ in $S(M_n)$ corresponding to the
configuration $u$ can be separated from 0 in a homomorphic image $E$ of $S(M_n)$ with at most $\Psi(n)-1$
elements. Then the images of $a_{K+1}$, $a_{K+2}$ in that semigroup satisfy $z^D=z^{2D}$ for some $D<T(n)$. Since
the halting computation has $>D$ steps, the letter $a_{K+1}$ occurs in $w(u)$ exactly once, and every command of
$M_n$ corresponding to a command of $M$ adds one coin in glass $K+1$, there exists a word $W$ which is equal to
$w(u)$ in $S(M_n)$ and which has the form $$q_ja_1^{l_1}\ldots a_k^{l_K}A_1\ldots A_Ka_{K+1}^DA_{K+1}A_{K+2}.$$
Modulo relations corresponding to the commands (\ref{c:a}), this word is equal to $$q_ja_1^{l_1}\ldots
a_k^{l_K}A_1\ldots A_Ka_{K+1}^{2D}A_{K+1}a_{K+2}^{D}A_{K+2}.$$ The image of the latter word in $E$ is equal to
$$q_ja_1^{l_1}\ldots a_k^{l_K}A_1\ldots A_Ka_{K+1}^{D}A_{K+1}a_{K+2}^{D}A_{K+2}$$ which, again modulo the
relations corresponding to the commands (\ref{c:a}), is equal to $$q_ja_1^{l_1}\ldots a_k^{l_K}A_1\ldots
A_KA_{K+1}A_{K+2}$$ which is equal to 0 by the relations corresponding to the commands (\ref{c:s}), a
contradiction.

Note that the co-time function of a Turing machine recognizing a recursive set can be larger than any given
recursive function. Indeed, after the machine halts without accepting, we can make it work as long as we like. It
remain to note that the co-time function of a Minsky machine simulating that Turing machine cannot be smaller.
\endproof

\section{Simulation of Minsky machines in solvable groups }\label{s:4}

Recall that a {\em variety} of algebraic structures is a class of all algebraic structures of a given type (signature) satisfying a given set of {\em identities} (also called {\em laws}). Equivalently, by a theorem of Birkhoff \cite{malcev} a variety is a class of algebraic structures closed under taking cartesian products, homomorphic images and substructures. Every variety contains free objects (called {\em relatively free} algebraic structures). One can define algebraic structures that are finitely presented in a variety as factor-structures by congruence relations generated by finite number of equalities. Every finitely presented algebraic structure which belongs to a variety $\mathcal V$ is finitely presented inside $\mathcal V$ but the converse is very rarely true. See \cite{KSs} for a survey of algorithmic problems for varieties of different algebraic structures (mostly semigroups, groups, associative and Lie algebras). In this section we concentrate on varieties of groups. The most well known varieties are the variety of Abelian groups $\mathcal A$ given by the identity $[x,y]=1$, the variety of nilpotent groups of class $c$, ${\mathcal N}_c$ given by the identity $[...[x_1,x_2],...,x_{c+1}]=1$, etc. The class of Abelian groups of finite exponent $d$, ${\mathcal A}_d$, is also a variety, given by two identities $[x,y]=1, x^d=1$.

If $\mathcal U$ and $\mathcal V$ are two varieties of groups then the class of groups consisting of extensions of groups from $\mathcal U$ by groups from $\mathcal V$ is again a variety (the {\em product} of $\mathcal U$ and $\mathcal V$) denoted by ${\mathcal U}{\mathcal V}$. The product of varieties is associative \cite{Neu}.
For example the variety of all solvable groups of class $c$ is the product of $c$ copies of the variety $\mathcal A$.
If $\mathcal V$ is a variety of groups, then ${\mathcal Z}{\mathcal V}$ is the variety consisting of all central extensions of groups from $\mathcal V$. For example ${\mathcal N}_2={\mathcal Z}{\mathcal A}$ and, more  generally,
${\mathcal N}_{c+1}={\mathcal Z}{\mathcal N}_c$ for every $c\ge 1$.

The problem of finding a finitely presented group with undecidable word problem, belonging to a proper variety of groups (i.e. satisfying a non-trivial identity) was formulated by Adian \cite{KT} and solved by the first author in \cite{Kh81}. The construction was simplified in the unpublished dissertation \cite{Kh90}. In this section, we shall modify the construction from \cite{Kh90} to construct residually finite finitely presented solvable groups with complicated word problem.

\subsection{The construction}
\label{s:construction}

Let $M$ be a Minsky machine with $K$ glasses and $N+1$ commands (numbered $0,...,N$). We are going to construct a
group $G(M)$ simulating $M$.   The group $G(M)$ will be in a sense similar to the semigroup $S(M)$ constructed
above. The main idea will be to replace the product by another operation and make sure that with respect to the
new operation the semigroup $S(M)$ ``embeds" into our group.

Thus the group will be generated by the $q$-letters which will be related to the letters $q_i$ from $S(M)$, and
also $a$-letters $a_1,...,a_K$, $A$-letters, $A_1,...,A_K$ and some extra $a$- and $A$-letters that help us
impose the necessary commutativity relations that, in particular, make the group solvable. The group we are going
to construct will be a semidirect product of the Abelian normal subgroup generated by the $q$-letters by the
semidirect product of an Abelian subgroup generated by $A$-letters and an Abelian subgroup generated by
$a$-letters. Thus we should have a way to ensure that in a subgroup generated by two sets of letters $Z\cup Y$,
the normal subgroup generated by $Z$ is Abelian. This is done with the help of the following lemma due to
Baumslag \cite{Baum} and Remeslennikov \cite{Rem}. In that lemma we denote $u^a=a\iv ua$ and $u^{a+b}=u^au^b$
(note that although $u^{a+b}$ is not necessarily equal to $u^{b+a}$, the equality will hold if the normal
subgroup generated by $u$ is Abelian, which is going to be the case every time we apply this lemma).

\begin{lemma}[\cite{Baum, Rem}] Suppose that a group $H$ is generated by three
sets $X, F=\{a_i\ |\ i=1,\ldots,m\}, F'=\{a'_i\ |\ i=1,\ldots,m\} $ such that

(1) The subgroup generated by $F\cup F'$ is Abelian;

(2) For every $a\in F$ and every $x\in X$ we have $x^{f(a)}= x^{a'}$ for some monic polynomial $f$ of $a$ which
has at least two terms (in all our applications $f(t)=t-1$);

(3) $[x_1^{a_1^{\alpha_1}\ldots a_m^{\alpha_m}},x_2]=1$,  for every $x_1,x_2\in X$, and every $\alpha_1,\ldots,
\alpha_m\in \{0,1,-1\}$.
\label{lBRG}

Then the normal subgroup generated by $X$ in the group $H=\la X\cup F\cup F'\ra$ is Abelian, and $H$ is
metabelian.
\end{lemma}
If the elements $a_i$ and $a'_i$ and the set $X$ satisfy the conditions of Lemma \ref{lBRG} we will call $a'_i$,
$i=1,...,m$, are BR-conjoints to $a_i$ with respect to $X$ (and the polynomial $f$).

Consider the free commutative monoid generated by letters $A_0,...,A_K$. Let $U_0$ be the set of all divisors of
the element $A_0A_1 ... A_K$ in that monoid, and $U$ be the set of all symbols $q_jw$, $w\in U_0, j=0,...,N$.
Also fix a prime $p$ (say, $p=2$).

The generating set of our group $G=G(M)$ will consists of three subsets:

$$L_0=\{x_{u}, u\in U,i=0,...,N\};$$ $$L_1=\{A_i,i=0,...,K,\};$$ $$L_2=\{a_i,a_i',\tilde a_i, \tilde a_i',
i=1,\ldots, K\}.$$

We introduce notation for some subgroups of the group $G$. Denote $H_i=\la L_i\ra, i=0,1,2.$ Denote also
$$M_0=\{\tilde a_i,\tilde a_i', A_0, i=1,\ldots,K\},\ M_i=\{a_i,{a_i'},A_i\}, i=1,\ldots,K$$

The group $G(M)$ has the following set of defining relations:

G1. Relations saying that $H_0$ and $H_1$ are Abelian groups of exponent $p$, and $H_2$ is an Abelian group.

G2. Any $y\in M_i$, $z\in M_j$, $i\ne j\in \{0,...,K\}$, commute.

G3. For every $i=1,\ldots,K$, $(a_i')^{-1}$ is a BR-conjoint to $a_i^{-1} $ with respect to $\{A_i\}$ (and
polynomial $f(t)=t-1$).

G4. The elements of the set $\{(\tilde a_i')^{-1}, i=1,...,K\}$ are a BR-conjoints to elements of the set
$\{\tilde a_i^{-1},i=1,...,K\}$ with respect to $\{A_0\}$.

G5. a) If $u\in U$ does not contain $A_i$ for some $i=0,\ldots,K$, then $[x_{u},A_i]=x_{uA_i}$, $j=0,...,K$.

b)  For every $i=1,\ldots,K$, if $u$ does not contain $A_i$, then $x_u^{a_i-1}=x_u^{a_i'}$ (see notation before
Lemma \ref{lBRG},

c) For every $i=0,\ldots,K$, if $u$ contains $A_i$,  $z\in M_i$, then   $[x_{u},z]=1$.

G6. $x_{q_j}^{a_i}=x_{q_j}^{\tilde a_i}$, $x_{q_i}^{a_i'}=x_{q_i}^{\tilde a_i'}$, $j=0,...,N$, $i=1,\ldots,K$.

G7. $[x_{u}^z,x_{v}]=1, $ where $z=a_1^{\alpha_1}\ldots,a_K^{\alpha_K}, \alpha_i\in\{-1,0,1\}.$

\begin{remark} \label{r22} Relations G7 together with G1 and G5b) imply that for every subset $I\subseteq
\{1,\ldots, K\}$ the
letters $\{a_i', i\in I\}$ are BR-conjoints of $\{a_i,i\in I\}$ with respect to  the set of all $x_{u}$'s where
$u$ does not contain letters $A_i, i\in I$.
\end{remark}

G8. Relations constructed from the program of the machine $M$. For every $f \in G$ denote $$f*a_i=
f^{-1}f^{a_i}f^{-a_i^{-1}}f^{(a_i')^{-1}}, i=1,...,K,$$ also let $$f*A_i=[f,A_i], i=0,\ldots, K. $$ We denote
$(...(t_1*t_2)*...)*t_k$ by $t_1*\ldots*t_k$, and $t_1*\underbrace{t_2*...*t_2}_{n}$ by $t_1*t_2^{(n)}$. The
relations corresponding to the commands of $M$ are in the following table.

\begin{equation}
\label{eq:s7}
\begin{array}{|l|l|}
\hline
\hbox{Command of } M& \hbox{Relation of }G(M) \\
\hline
&\\ i\rightarrow \Add(n_1,\ldots,n_m);j & x_{q_iA_0}=x_{q_jA_0}*a_{n_1}*...*a_{n_m}\\
\hline
&\\ i,\epsilon_{n_1}>0,\ldots,\epsilon_{n_m}>0\to \Sub(n_1,\ldots,n_m);j & x_{q_iA_0}*a_{n_1}*\ldots *
a_{n_m}=x_{q_jA_0}\\
\hline
 &\\
i,\epsilon_{n_1}=0,\ldots,\epsilon_{n_m}=0\to j  & x_{q_iA_0}*A_{n_1}*\ldots * A_{n_m}\\ & =
x_{q_jA_0}*A_{n_1}*\ldots *A_{n_m}\\
\hline
\end{array}
\end{equation}

\begin{theorem} \label{tmm} (a) The group $G(M)$ belongs to ${\mathcal A}_p^2{\mathcal A}\cap
{\mathcal Z}{\mathcal
N}_{K+1}{\mathcal A}$.

(b) The equality $$x_{q_iA_0}*a_1^{(m_1)}*\ldots a_K^{(m_K)}*A_1^{(\alpha_1)}*\ldots
*A_K^{(\alpha_K)}=x_{q_jA_0}*a_1^{(n_1)}*\ldots*a_K^{(n_K)}*A_1^{(\beta_1)}*\ldots A_K^{(\beta_K)}$$ where
$\alpha_i,\beta_i\in \{0,1\}$ is true in $G(M)$ if and only if the equality $$q_ia_1^{m_1}\ldots
a_K^{m_K}A_1^{\alpha_1}\ldots A_K^{\alpha_K}=q_ja_1^{n_1}\ldots a_K^{m_K}A_1^{\beta_1}\ldots A_K^{\beta_K}$$ is
true in the semigroup $S(M)$ (in particular, $\alpha_i=\beta_i$ for every $i$).

(c) The equality $$x_{q_i}*a_1^{(m_1)}*\ldots a_K^{(m_K)}*A_1^{(\alpha_1)}*\ldots
*A_K^{(\alpha_K)}=x_{q_j}*a_1^{(n_1)}*\ldots*a_K^{(n_K)}*A_1^{(\beta_1)}*\ldots A_K^{(\beta_K)}$$ where
$\alpha_i,\beta_i\in \{0,1\}$ is true in $G(M)$ if and only if $m_i=n_i, \alpha_i=\beta_i$ for every $i$.
\end{theorem}

\begin{proof} First we will prove part (a): $G(M)\in  {\mathcal A}_p^2{\mathcal A}\cap {\mathcal Z}{\mathcal N}_3{\mathcal A}$.

\begin{lemma}\label{lH} The subgroup $\la H_1\cup H_2\ra$ of $G$ is metabelian and a semidirect product of
the Abelian normal subgroup $H_1^{H_2}$ of exponent $p$, and $H_2$. \end{lemma}
\begin{proof}
Indeed by relations G2, $$\la M_i,i=0,...,K\ra=\prod_{i=1}^K\la a_i,a_i',A_i\ra\times \la\tilde a_i, \tilde a_i',
A_0, i=1,\ldots,K\ra.$$ Using relations G1, G3, G4, we can apply Lemma \ref{lBRG} to each of the factors in that
direct product and conclude that each of them is metabelian and a semidirect product of the Abelian of exponent
$p$ normal subgroup generated by the intersection of $\{A_i, i=0,\ldots, K\}$ with that factor, and the Abelian
group generated by the $a$-letters from that factor.\end{proof}

\begin{lemma}\label{lT}The normal subgroup $T$ of $G$ generated by all the elements $x_{u}, u\in U,$ is Abelian
of exponent $p$.
\end{lemma}
\begin{proof}  Relations G5 a) of the group $G$ imply that every element $x_u, u\in U$ is a product of elements
$x_{q_j}^z, z\in H_1, i=0,\ldots ,N.$ Therefore, it is enough to show that
\begin{equation}\label{ab}
x_{q_k}x_{q_t}^z=x_{q_t}^zx_{q_k}\end{equation} for any $z\in \la H_1,H_2\ra$ and any $k,t$. To reduce the proof
of these equalities to the proof of more simple equalities notice that $z=z_0z_1\ldots z_K$ where $z_i\in \la
M_i\ra$ by G2. Therefore equalities  (\ref{ab}) are equivalent to
\begin{equation}\label{ab1}
x_{q_k}^{z_0}x_{q_t}^{z_1\ldots z_K}=x_{q_t}^{z_1\ldots z_K}x_{q_k}^{z_0}.\end{equation} We can represent element
$x_{q_j}^{z_i}$, $i\ge 1$, as a product of elements of the form $x_{q_j}^{a_i^{p}{(a_i')}^{q}}$ and
$x_{q_jA_i}^{{\tilde a_i}^{p}{(\tilde a_i')}^{q}}.$ Indeed we have the following sequence of equalities deduced
using G2, G5, G6:

\begin{equation}\label{e00} \begin{array}{l} x_{q_j}^{a_i^{r_1}{(a_i')}^{s_1}A_i^{t_1}\ldots
a_i^{r_s}{(a_i')}^{s_k}A_i^{t_k}} \stackrel{\hbox{\tiny{G6}}}{=\!\!=\!\!=} x_{q_j}^{\tilde a^{r_1}
 {(\tilde a_i')}^{s_1}A^{t_1}\ldots a_i^{r_k}{(a_i')}^{s_k}A_i^{t_k}}\\ \\
 \stackrel{\hbox{\tiny G2}}{=\!\!=\!\!=}x_{q_j}^{A_i^{s_1}\tilde a_1^{r_1}
 {(\tilde a_i')}^{s_1}a_i^{r_2}{(a_i')}^{s_2}\ldots a_i^{r_k}{(a_i')}^{s_k}A_i^{t_k}}\\ \\
 \stackrel{\hbox{\tiny G5  a), c), G6}}{=\!\!=\!\!=\!\!=\!\!=\!\!=\!\!=}
 x_{q_j}^{a_i^{r_1+r_2}{(a_i')}^{s_1+s_2}A_i^{t_2}\ldots
 a_i^{r_k}{(a_i')}^{s_k}A_i^{t_k}}(x_{q_jA_i}^{t_1})^{\tilde
 a_i^{r_1}(\tilde a_i')^{s_1}}=\\ \\ \ldots =
 x_{q_j}^{a_i^{r_1+\ldots +r_k}{(a_i')}^{s_1+\ldots + s_k}}(x_{q_tA_i}^{t_{k}})^{\tilde a_i^{r_1+\ldots
 r_{k}}(\tilde
 a_i')^{s_1+\ldots s_{k}}}\ldots (x_{q_tA_i}^{t_1})^{\tilde a_i^{r_1}(\tilde
 a_i')^{s_1}}.\end{array}\end{equation}
Repeating this argument $K$ times, one proves that $x_{q_j}^{z_1z_2\ldots z_K}$ can be represented as a product
of elements of the form $x_u^y$ where $u\in U$, $y\in H_2$. A similar proof (using also G4) gives that
$x_{q_j}^{z_0}$ is a product of elements of that form. It remains to note that elements of the form $x_u^y$,
$u\in U, y\in H_2$ commute by Remark \ref{r22} and Lemma \ref{lBRG}.
\end{proof}

\begin{remark}\label{r4} Note that equalities (\ref{e00}) and similar equalities when $x_{q_j}$ is replaced by
$x_u$, $u\in U$, imply the following: if $y$ is a product of elements of the form $a_i^{r_l}(a_i')^{s_l}A_i$ and
$\sum_l r_l=\sum_l s_l=0$, then $[x_u,y]$ is 1 if $u$ contains $A_i$ or a product of conjugates of elements
$x_{uA_{i}}$ by elements from $\la \tilde a_i\ra\times \la \tilde a_i'\ra$ otherwise. Similarly, suppose that $y$
is a product of elements from $M_0$, each factor containing $A_0$,  and the total exponent of every $\tilde a_i$
(resp. $\tilde a_i'$) is 0. Then $[x_u,y]=1$ provided $u$ contains $A_0$ and is a product of conjugates of
$x_{uA_0}$ by elements from $\la  a_i, a_i'\ra$ provided $u$ does not contain $A_0$.
\end{remark}

By construction, the group $G$ is a semidirect product of $T$ and the metabelian group $H_1^{H_2}\rtimes H_2$. By
Lemma \ref{lH}, $G$ is solvable of class 3 and, moreover, belongs to ${\mathcal A}_p^2{\mathcal A}$.

\begin{remark}\label{r8} The proof of Lemma \ref{lT} shows that $T$ is generated (as an Abelian group) by
elements of the form $x_u^y$ where $u\in U$ and $y\in H_2$.
\end{remark}

\begin{lemma}\label{l9}  The quotient of $G(M)$ over the center satisfies the identity
$$[[x_1,y_1],[x_2,y_2],\ldots,[x_{K+2},y_{K+2}]]=1.$$ This means that $G$ belongs to the variety ${\mathcal
Z}{\mathcal N}_{K+1}{\mathcal A}$.
\end{lemma}
\begin{proof} Let $P$ be the derived subgroup of $G(M)$. By Lemma \ref{lT}, every element of $P$ is a product of
an element of $T$ and an element of $H_1^{H_2}$. It also follows from Lemma \ref{lT} that $[P,P]\subseteq T$,
hence by Remark \ref{r8}, it is generated by elements of the form $x_u^y$, $u\in U, y\in H_2$, the word $u$
contains at least one $A_i$, $i=0,\ldots,K$. Since $T$ is Abelian, the subgroup $\underbrace{[P,P,\ldots,
P]}_{K+2}$ is generated by the commutators $$[x_u^y,h_{1,1}^{h_{2,1}},...,h_{1,K}^{h_{2,K}}]$$ for some
$h_{1,i}\in H_1$, $y,h_{2,i}\in H_2$. An easy induction shows that every such commutator is a conjugate of
\begin{equation}\label{e7}
[x_u,h_{1,1}^{y'},...,h_{1,K}^{y'}]
\end{equation}
where $y'\in H_2$.

Let $h\in H_1$, $u\in U, y\in H_2$. Suppose that $h=A_{i_1}^{t_1}\cdot...\cdot A_{i_s}^{t_s}$ where $t_i\not=0$.
Consider $[x_u,h^y]$. Then Remark \ref{r4} implies that $[x_u,h^y]$ is a product of elements of the form
$x_{u'}^{y'}$ where $u'\in U$ contains letters $A_{i_1},...,A_{i_s}$ and it may not be equal to 1 only if one of
the letters $A_{i_j}$ does not occur in $u$. Therefore the commutator (\ref{e7}) is either equal to 1 or is a
product of elements of the form $x_{u'}^{y''}$ where the word $u'\in U$ contains all letters $A_0, A_1,...,A_K$,
$y''\in H_2$. But every such $x_{u'}$ is in the center of $G(M)$ by G5 c). Hence
$\underbrace{[P,\ldots,P]}_{K+2}$ is contained in the center of $G(M)$. \end{proof}

We now prove (b) and (c). For this, as we mentioned before Lemma \ref{forward}, we need to prove Lemmas
\ref{forward} and \ref{llemma2}. Lemma \ref{forward} for $G(M)$ is proved in the same way as for the semigroup
$S(M)$ (see \cite{Sapir,KSs},since the only property of $S(M)$ used there was that the word $w=q_ia_1^{l_1}\ldots
a_K^{l_K}A_1^{\alpha_1}\ldots A_K^{\alpha_K}$ is equal to any word obtained from $w$ by permuting $a_i$ with
$a_j$, $A_i$ with $A_j$ and $a_i$ with $A_j$ ($i\ne j$). The same is true for words of the form
\begin{equation}\label{e000} x_{q_iA_0}*a_1^{(m_1)}*\ldots a_K^{(m_K)}*A_1^{(\alpha_1)}*\ldots
*A_K^{(\alpha_K)}\end{equation} in $G(M)$ by
the definition of the operation $*$, relations G1, G2 and Lemma \ref{lT}.

In order to prove Lemma \ref{llemma2} we will define a new group $\bar G$ that is a quotient of $G$ and injective
on elements of the form (\ref{e000}).

Let $\check S$ be the semigroup with the same generating set as $S(M)$ subject all the relations of $S(M)$ except
the relations (\ref{eq:s2}) corresponding to the commands of $M$ (that semigroup does not depend on $M$). Thus
non-zero elements in $\check M$ have the form $$ q_i^{\alpha_1}a_1^{l_1}\ldots a_K^{l_K}A_1^{\alpha_1}\ldots
A_K^{\alpha_K} $$ where $l_j\in \N$, $\alpha_j\in \{0,1\}$. Let $W$ be the set of all non-zero elements of
$\check S$ containing a $q$-letter, and $W_0$ be the set of elements from $W$ viewed as elements of $S(M)$ (i.e.
different words may represent equal element) with $A_0$ inserted next to the $q$-letter. Consider the free
Abelian group $T_1$ of exponent $p$ generated by the elements $z_{i_1,\ldots,i_K,u}$, $u\in W\cup W_0$, $i_j\in
\{1,2,3\}$. For each element of $L_1\cup L_2$, we define an automorphism of $T_1$. The group $\bar G$ will be the
semidirect product of $T_1$ and the group generated by these automorphisms.

For simplicity we will denote automorphisms corresponding to letters from $L_1\cup L_2$ by the same letters.

Let us start with automorphisms $a_j,\ a_j'$. We have to define $z_{i_1,\ldots,i_K,u}^{a_i}$ and
$z_{i_1,\ldots,i_K,u}^{a'_i}$ for every $i_1,...,i_K$. First suppose that $u$ does not contain $A_j$. To simplify
the notation we shall denote the vector $(i_1,\ldots,i_K)$ by $\vec i$, and the standard unit vectors by $\vec
e_l, l=1,...,K$. We shall write $z_{\vec i, u}$ instead of $z_{i_1,\ldots, i_K, u}$. The $j$-th coordinate of
$\vec i$ is denoted by $\vec i_j$.

\begin{equation}\label{ea}
z_{\vec i,u}^{a_j}  = \left\{\begin{array}{ll} z_{\vec i,u}z_{\vec i+\vec e_j,u}z_{\vec i+2\vec e_j,u}z_{\vec
i,ua} & \hbox{ if } \vec i_j=1;\\ z_{\vec i,u}z_{\vec i-\vec e_j,u}^{-1} & \hbox{ if } \vec i_j=2;\\ z_{\vec
i-2\vec e_j,u}& \hbox{ if } \vec i_j=3.
\end{array}\right.
\end{equation}
$$ z_{\vec i,u}^{a_j'} = z_{\vec i,u}^{-1}z_{\vec i,u}^{a_j}. $$

\vskip 0.1 in

If $u$ contains letter $A_j$, then let $z_{\vec i,u}^{a_j}=z_{\vec i,u}^{a_j'}=z_{\vec i,u}.$

\vskip 0.1 in

It is easy to prove that $a_j$ is an automorphism by constructing the automorphism $a_j\iv$. If we apply
$a_j^{-1}$ to the third equality in (\ref{ea}), we will obtain the formula for $z_{\vec i,u}^{a_j\iv}$ provided
$\vec i_j=1$ (and $u$ does not contain $A_j$). Plugging it in the second equality of (\ref{ea}) we obtain the
formula for $z_{\vec i,u}^{a_j\iv}$ provided $\vec i_j=2$. Finally plugging it in the first equality in
(\ref{ea}), we obtain the formula for  $z_{\vec i,u}^{a_j\iv}$ provided $\vec i_j=3$:

$$ x_{\vec i, u}^{a\iv}=
\left\{\begin{array}{l}
x_{\vec i-\vec e_j,u}^{-1}x_{\vec i,u}^{-1}x_{\vec i,ua}^{-1}, \hbox{ if }i_j=3\\
\\
x_{\vec i,u}x_{\vec i+\vec e_j,u},\hbox{ if } i_j=2\\
\\
x_{\vec i+2\vec e_j,u}, \hbox{ if } i_j=1.
\end{array}\right.
$$

The automorphism $\tilde a_j$ is defined similarly. If $u$ contains $A_0$, then $z_{\vec i,u}^{\tilde
a_j}=z_{\vec i, u}$. If $u$ does not contain $A_0$ and $A_j$ then $z_{\vec i,u}^{\tilde a_j}=z_{\vec
i,u}^{a_j}.$

If $u$ does not contain $A_0$ but contains $A_j$, i.e. $u=vA_j$ for some $v$, then $$ z_{\vec i,u}^{\tilde a_j}
=\left\{\begin{array}{ll} z_{\vec i,u}z_{\vec i+\vec e_j,u}z_{\vec i+2\vec e_j,u}z_{\vec i,va_jA_j}, & \hbox{ if
}
\vec i_j=1;\\ z_{\vec i,u}z_{\vec i-\vec e_j,u}^{-1}, & \hbox{ if } \vec i_j=2;\\
z_{\vec i-2\vec e_j,u},& \hbox{ if }\vec i_j=3.\end{array}\right.\\ $$ $$ z_{\vec i,u}^{\tilde a_j'} = z_{\vec
i,u}^{-1} z_{\vec i,u}^{\tilde a_j}. $$

Finally the automorphisms corresponding to $A_j$, $j=0,...,K$, are defined as follows: $$z_{\vec
i,u}^{A_j}=z_{\vec i,u}z_{\vec i,uA_j}$$ if $u$ does not contain $A_j$ and $$z_{\vec i,u}^{A_j}=z_{\vec i,u}$$ if
$u$ contains $A_j$.

The following lemma is obtained by a straightforward application of the definition of the automorphisms above and
the definition of the operation $*$.  This lemma implies that $\bar G$ satisfies G8 if we replace $x_u$ by
$z_{\vec 1, u}$ (since the corresponding relations hold in $S(M)$).

\begin{lemma}\label{l654} The following relations hold in $\bar G$. For every $w\in \{a_j, A_j, j=1,...,K\}$

$$ z_{\vec 1,u}*w=z_{\vec 1,uw} $$
 where we set $z_{\vec 1,0}=1$ (the identity element in $G(M)$) where $*$ is defined in G8.
\end{lemma}

We define $\bar G$ as the semidirect product of $T_1$ and the subgroup of $\Aut(T_1)$ generated by the
automorphisms corresponding to the elements from $L_1\cup L_2$. From the definition of the automorphisms and
Lemma \ref{l654}, it follows that $\bar G$ is generated by the elements $z_{\vec 1,u}$, $u\in U$, where $\vec 1$
is the vector $(1,1,\ldots,1)$ and the automorphisms corresponding to elements of $L_1\cup L_2$. It is easy  to
check that all the relations G1-G8 hold in $\bar G$, therefore
\begin{lemma}\label{l789} The map that sends every $a$- or $A$-letter to itself and every $x_{u}$ to $z_{\vec
1,u}$ extends to a homomorphism $\phi$ from $G$ to $\bar G$. \end{lemma}

\begin{lemma}\label{lphi} The homomorphism $\phi$ is surjective.
\end{lemma}

\proof It is easy to see that we only need to define pre-images $x_{\vec i, w}$ of elements $z_{\vec i,u}\in
\bar
G$, $w\in W\cup W_0$. By the definition of $\phi$, we have $\phi(x_u)=z_{\vec 1,u}$ for every $u\in U$ so we
define $x_{\vec 1, u}=x_{u}$. The other preimages are defined by induction  on the length of $w$ and the sum of
$\vec i_j$.

Suppose $w\in W\cup W_0$ does not contain $A_j$ and $\vec i_j=1$, $\vec i'$ is arbitrary. Then we define:
$$\begin{array}{l}x_{\vec i+\vec e_j,w}=x_{\vec i,w}^{-(a_j')^{-1}},\\ x_{\vec i+2\vec e_j,w}=x_{\vec i,
w}^{a_j^{-1}},\\ x_{\vec i',w}*a_j=x_{\vec i',wa_j}.\end{array}$$ We also have $x_{\vec i,w}*A_j=x_{\vec i,wA_j}$
for any $\vec i.$

It is easy to see that for every $\vec i$ and $w\in W\cup W_0$, we have $\phi(x_{\vec i, w})=z_{\vec i, w}$. This
proves the lemma.
\endproof

In $\bar G(M)$, consider the set $P$ of elements
\begin{equation}\label{e:p}z_{\vec 1,q_iA_0}*a_1^{(m_1)}*\ldots a_K^{(m_K)}*A_1^{(\alpha_1)}*\ldots
*A_K^{(\alpha_K)}
\end{equation}
where $\alpha_i\in\{0,1\}$ and the set $P_0$ of elements
\begin{equation}\label{e:p0}z_{\vec 1,q_i}*a_1^{(m_1)}*\ldots a_K^{(m_K)}*A_1^{(\alpha_1)}*\ldots
*A_K^{(\alpha_K)}
\end{equation}
By construction $P\cap P_0=\emptyset$, elements (\ref{e:p}) are different if and only if elements
$$q_ia_1^{m_1}\ldots a_K^{m_k}A_1^{\alpha_1}\ldots A_K^{\alpha_K}$$ from $S(M)$ are different,  and elements
$(\ref{e:p0})$ are different if and only if the corresponding elements $q_ia_1^{m_1}\ldots
a_K^{m_k}A_1^{\alpha_1}\ldots A_K^{\alpha_K}$ of $\check S$ are different. This completes the proof of Lemma
\ref{llemma2}and Theorem \ref{tmm} (b), (c).
\end{proof}

We shall need a few more properties of the group $G(M)$.

\begin{lemma}\label{l098} Let elements $x_{\vec i,w}, w\in W\cup W_0$ from $G$ be defined as in the proof of
Lemma \ref{lphi}. Let $y\in L_1\cup L_2$, $w\in W\cap W_0$. Then for every $i\in \{1,2,3\}^{\{1,\ldots,K\}}$,
$x_{\vec i,u}^y$ satisfies the same equalities as elements $z_{\vec i,w}^y$ from the definition of automorphism
of $\bar G$ with $z$ replaced by $x$ everywhere. In particular, $x_{\vec i,u}^y$ is a product of one or several
elements of the form $x_{\vec i',w'}$ such that every letter $a_j$ occurs in $w'$ at least as many times as in
$w$ (in particular if for some $R>0$, $w$ belongs to the ideal $V_R$ defined in Lemma \ref{l:fd1}, then $w'\in
V_R$.
\end{lemma}

\proof For $y\in \cup M_j, j\ge 1,$ this follows from the way $x_{\vec i,u}$ are constructed. For $y\in M_0$,
one
needs to use G2, G5 c), and G6.
\endproof

The proof of Lemma \ref{l098} actually gives the following

\begin{lemma}\label{l099} If $v$ is a word in $a$- and $A$-letters (i.e. over $L_1\cup L_2$), then $x_{\vec i,
u}^v$ is a product in $G$ of elements $x_{\vec j, w} $ as in Lemma \ref{l098} where the length of each $w$ does
not exceed the length of $v$ (hence the total number of different $x_{\vec j, w}$ occurring in this product is
polynomial in terms of $|v|$.
\end{lemma}

\begin{lemma}\label{l097} The normal subgroup $T$ generated by the elements $x_{u}, u\in U$  in $G=G(M)$ is the
direct product of cyclic subgroups generated by the elements $x_{\vec i,w},\vec i\in \{1,2,3\}^{\{1,\ldots,K\}},
w\in W\cup W_0.$
\end{lemma}

\begin{proof} By Lemma \ref{l098} elements $x_{\vec i,w}$ span $T$.  We defined elements $x_{\vec i,w}, w\in W$
in such a way
that they are  pre-images of the corresponding elements  in $\bar G$ under $\phi$. Thus the elements $x_{\vec
i,w}, \vec i\in
\{1,2,3\}^{\{1,\ldots,K\}}, w\in W\cup W_0$ are linearly independent since their images under $\phi$ are
linearly
independent in $T_1$.
\end{proof}

\begin{lemma}\label{l096} Let $R>0$,  $V_R$ be the ideal of the semigroup $S(M)$ defined in Lemma \ref{l:fd1}.
Then the subgroup $T(V_R)$ of $T$ spanned by all the elements $x_{\vec i,w}$, $w\in V_R, \vec i\in
\{1,2,3\}^{\{1,\ldots,K\}}$ is normal in $G(M)$ (as before, we set $x_{\vec i, 0}=1$) and of finite index in
$T$.
\end{lemma}

\proof The first part follows from Lemma \ref{l098}. If $\{w_1,...,w_M\}$ is the set $S(M)\setminus V_R$, then
$\{x_{\vec i, w}, \vec i\in \{1,2,3\}^{\{1,\ldots,K\}}, w\in \{w_1,...,w_M\}\}$ is a set of representatives of all
cosets of $T(V_R)$ in $T$.
\endproof

\subsection{A finitely presented solvable group with undecidable word problem}

By Theorem \ref{t:MM}, there exists a 2-glass Minsky machine which computes a non-recursive partial function. The
corresponding group $G(M)$ has undecidable word problem and belongs to the variety ${\mathcal A}_p^2{\mathcal A}\cap
\mathcal{ZN}_3\mathcal{A}$ by Theorem \ref{tmm}. Hence we obtain the following:

\begin{theorem} [Kharlampovich  \cite{Kh81}] There exists a finitely presented group with undecidable word
problem that belongs to the variety ${\mathcal A}_p^2{\mathcal A}\cap
\mathcal{ZN}_3\mathcal{A}$.
\end{theorem}

\subsection{Residually finite finitely presented groups}

\begin{theorem}\label{t:rfg1} If a Minsky machine $M$ is sym-universally halting then  the group $G(M)$ is
residually finite. Its word problem is at least as hard as the halting problem for $M$.
\end{theorem}

\begin{proof}  Let $M$ be a sym-universally halting Minsky machine. Let $w\ne 1\in G(M)$. We use the notation
from the definition of $G(M)$. There exists a natural homomorphism $\zeta$ from $G(M)$ to the metabelian group
$H_1^{H_2}\rtimes H_2$ which kills all elements from $T$. Since every finitely generated metabelian group is
residually finite, we can assume that $\zeta(w)=1$. Hence $w\in T$. By Lemma \ref{l097}, $x$ is a product of
elements of the form
\begin{equation}\label{e:x}
x_{\vec i,u},\vec i\in \{1,2,3\}^{\{1,\ldots,K\}}, u\in W\cup W_0.
\end{equation} Hence $w=w_0w_1$ where $w_0$ (resp. $w_1$) is a product of elements (\ref{e:x}) with $u\in W_0$
(resp. $W$). Suppose that $w_1$ is not 1. Let $T'$ be the subgroup of $G(M)$ generated by elements (\ref{e:x})
with $u\in W_0$. Then $T'$ is a normal subgroup of $G(M)$ by Lemma \ref{l098}. Let $G'(M)=G(M)/T'$. This group is
a semidirect product of $T/T'$ and the metabelian group $H_1^{H_2}\rtimes H_1$.  Let $D$ be the sum of lengths of
words $u\in W$ that appear in the factors of $w_1$. Let $Y_D$ be the set of all words in $\check S$ where at
least one $a$-letter appears at least $D$ times, and $0$. Then $Y_D$ is an ideal in $\check S$, and the image of
the set of elements (\ref{e:x}) with $u\in Y_D$ in $G'(M)$ form a normal $N$ subgroup of $G'(M)$ of finite index
(because $T$ is an Abelian group of finite exponent $p$). That normal subgroup does not contain $w$ by Theorem
\ref{tmm} (c). Then $G'(M)/N$ is a semidirect product of a finite group and the metabelian group
$H_1^{H_2}\rtimes H_2$. Hence $G'(M)/N$ is residually finite and $w$ can be separated from 1 by a homomorphism
from $G(M)$ onto a finite group.

Finally suppose that $w_1=1$. Let $u_1,...,u_l$ be the elements from $W_0$ that appear in the representation of
$w$ as a product of elements (\ref{e:x}). Let $E$ be the set of words that is equal to one of the $u_j$ in
$S(M)$. Since $M$ is sym-universally halting, $E$ is finite. Let $D$ be the maximal length of a word in $E$. Let,
as above, $Y_D$ be the ideal in $\check S$ consisting of 0 and all elements where one of the $a$-letters appears
at least $D$ times. Let $Z_D$ be the set of non-zero elements of $S(M)$ that are images of words from $Y_D$ under
the natural homomorphism $\check S\to S(M)$. Then $Z_D$ does not contain $u_1,...,u_l$. Consider the subgroup $F$
of $T$ spanned by all elements (\ref{e:x}) with $u\in Z_D\cup Y_D$. From Lemma \ref{l098}, it follows that $F$ is
a normal subgroup of $G(M)$ of finite index in $T$. Since $Z_D$ does not contain $u_1,...,u_l$, the subgroup $F$
does not contain $w$. The factor-group $G(M)/F$ is a semidirect product of a finite group and the metabelian
group $H_1^{H_2}\rtimes H_2$, and we can complete the proof as above.
\end{proof}

\begin{theorem}\label{t:rfg}For every recursive function $f$, there is a  residually finite finitely presented
solvable of class 3 group $G$ with Dehn function greater than $f$. In addition, one can assume that the word problem in $G$ is at least as hard as the membership problem in a given recursive set of natural numbers $Z$ or as easy as polynomial time.
\end{theorem}

\begin{proof}
The  statement follows from Theorems \ref{t:rfg1}  and \ref{t:rc}.
\end{proof}

\subsection{Residually finite finitely presented group with large depth function}

\begin{theorem}\label{t:rhog} For every recursive function $f$ there exists a finitely presented residually
finite group $G$ from ${\mathcal A}_p^2{\mathcal A}\cap {\mathcal Z}{\mathcal N}_3{\mathcal A}$ such that $\rho_G(n)>f(n)$ for
all $n$. In addition, we can assume that the word problem in $G$ is as hard as the membership problem for any
prescribed recursive set of natural numbers.
\end{theorem}

\proof Consider the Minsky machine $M_n$ constructed in the proof of Theorem \ref{t:rho}. Then  as in the proof
of Theorem \ref{t:rfg}, one can prove that $G(M_n)$ is residually finite. The fact that $\rho_G(n)>f(n)$ is
proved the same way as in the proof of Theorem \ref{t:rho} (one only needs to replace the product by operation *
everywhere in that proof).
\endproof

\section{Distortion of subgroups closed in the pro-finite topology}\label{s:dis}

Let us generalize the Mikhailova construction \cite{Mikhailova}.

Let $G$ be a finitely generated group generated by a finite set $X$, $N \leq G$ a normal subgroup, generated as a
normal subgroup by a finite set $R = \{r_1, \ldots, r_k\}$, and $\phi:G \to G/N$ the canonical epimorphism. We
may assume that both sets $X$ and $R$ are symmetric, i.e., $X=X^{-1}$ and $R=R^{-1}$. The set $$ E(G,N) = \{
(u,v) \in G \times G \mid \phi(u) = \phi(v)\} $$ is a subgroup of $G \times G$, called the {\em equalizer} of
$(\phi,\phi)$.

In the following lemma we summarize the main components of Mikhailova's argument (though in a much more general
situation). The proof is easy and we leave it to the reader.

\begin{lemma}
In the notation above the following hold:
\begin{itemize}
\item $E(G,N)$ is generated by a finite set
$$D=\{(r,1)\mid r\in R\}\cup\{(x,x^{-1})\mid x\in X\} \subset  \ G \times G.$$

\item For any $w\in G$ if
    $(w,1)= (u_1,v_1)(u_2,v_2)\ldots (u_n,v_n) \mbox{ for some } (u_i,v_i)\in D,$
then $u_1\ldots u_n$ is of the form
    $w_0 r_1 w_1 r_2 w_2 \ldots w_{m-1}r_m w_m \mbox{ for some } w_i\in G,\ r_i\in R, m \leq n,$
satisfying $w_0w_1\ldots w_m =1$ in $G$, hence,
    $$w =_{G} \prod_{i=1}^m (w_0\ldots w_{i-1}) r_i (w_0\ldots w_{i-1})^{-1}.$$
    In particular, the distortion of $E(G,N)$ in $G \times G$ is at least as high as the Dehn function of
    $G/N$ relative to $N$.
\end{itemize}
\end{lemma}

Let  $\mathcal P$  be a class of finite groups closed under direct products and subgroups. Recall that the pro-$\mathcal P$ topology on a group $G$ has as its base the set of all normal subgroups $N$ with $G/N\in \mathcal P$.

\begin{lemma}\label{l:89}
Let  $\mathcal P$  be a class of finite groups closed under direct products and subgroups. In the notation above if the group
$G/N$ is residually $\mathcal P$  then the subgroup $E(G,N)$ is closed in the pro-$\mathcal P$ topology on $G \times
G$.
\end{lemma}
\begin{proof}
Suppose $(u,v) \in G \times G$ but $(u,v) \not \in  E(G,N)$, so $\phi(u) \neq \phi(v)$. Since $G/N$ is residually
$\mathcal P$ there is a homomorphism $\eta:G/N \to K$ onto a finite group $K \in \mathcal P$ such that
$\eta\phi(u) \neq \eta\phi(v)$ in $K$. Therefore the image of the pair $(u,v)$ under $\eta\phi$ is not in the
image of the subgroup $E(G,N)$ in $K \times K$.  Hence the subgroup $E(G,N)$ is closed
in the pro-$\mathcal P$ topology on $G \times G$.
\end{proof}

The same argument gives the following

\begin{lemma}\label{l:90} Under the assumptions of Lemma \ref{l:89}, the relative depth function $\rho_{E(G,N)}$
is at least as large as the depth function of $G/N$, the time complexities of the ``yes" and ``no" parts of the
membership problem for $E(G,N)$ are as high as the time complexities of the ``yes" and ``no" parts of the word
problem in $G/N$.
\end{lemma}

Lemma \ref{l:90} and Theorem \ref{t:rfg} imply

\begin{remark} \label{r:m} The converse of Lemma \ref{l:89} also holds, namely, if $E(G,N)$ is closed in the pro-$\mathcal P$-topology, then $G/N$ is residually ${\mathcal P}$. We are not using this remark below so we leave it as an (easy) exercise.
\end{remark}

\begin{theorem} \label{t:distortion}
For any recursive function $f(n)$ there is a finitely generated  subgroup $H \le F_2 \times F_2$ such that $H$ is
closed in the pro-finite topology on $F_2 \times F_2$ and has distortion at least $f(n)$.
\end{theorem}
\begin{proof}
Let $G = \langle X \mid R\rangle$ be a finitely presented residually finite group with Dehn function at least
$f(n)$ from Theorem \ref{t:rfg}. If $N$ is the normal closure of $R$ in $F(X)$ then the subgroup $H=E(F(X),N)
\leq F(X) \times F(X)$ satisfies all the requirements of the theorem.

Now one can embed the free group $F(X)$  into $F_2$ in such a way that the pro-finite topology induced on the
image of $F(X)$ from $F_2$ is precisely the pro-finite topology on $F(X)$. Indeed, there is a finite index
subgroup $H$ of $F_2$ of rank $|X|$, the induced topology on $H$ is the pro-finite topology on $H$. It follows that the pro-finite topology on the subgroup $F_{|X|}$ of $F_2$ is precisely the topology
induced by the pro-finite topology from $F_2$, as required.
\end{proof}

Applying the same argument to the free solvable groups $S_3(X)$ of class 3 and generating set $X$ one gets the
following result.

\begin{theorem}\label{t:distS}
For any recursive function $f(n)$ there is a finite set $X$ and a finitely generated  subgroup $H \le S_3(X)
\times S_3(X)$ such that $E$ is closed in the pro-finite topology on $S_3(X) \times S_3(X)$ and has distortion
function, relative depth function, the time complexities of both ``yes" and ``no" parts of the membership problem and at least
$f(n)$.
\end{theorem}

\section{Universal theories of sets of finite solvable groups}\label{s:6}

In this section we will prove the following result. For the class of all finite groups in was proved by
Slobodskoi \cite{sl} (the idea of Slobodskoi's proof came from Gurevich's paper \cite{Gurevich} where the same result was proved for semigroups).

\begin{theorem}\label{t:sl} The universal theories of the class of finite groups from $\mathcal{A}_p^2{\mathcal A}\cap
{\mathcal Z}{\mathcal N}_5{\mathcal A}$ and the class of all periodic groups are recursively inseparable. In
particular, the universal theory of any set of finite groups containing all finite solvable of class 3 groups is
undecidable.\end{theorem}

\begin{proof}
It is well known \cite{Gurevich} that there exists a Turing machine for
which the set of input configurations accepted by the machine and the set of input configurations starting with
which the machine never stops are recursively inseparable. Let $M$ be a 2-glass Minsky machine with the same
property.

Consider the 4-glass Minsky machine $M_n$ described in the proof of Theorem \ref{t:rho}. Let $S'(M_n)$ be the
semigroup given by the same defining relations as $S(M_n)$ except the relation $q_0=0$ is substituted by the
relation $q_iA_{3}A_4=0$ for every $i$. It does not affect the proof of Theorem \ref{tmm}.

Let $G'(M_n)$ be the group corresponding to $S'(M_n)$ in the same way $G(M_n)$ corresponds to $S(M_n)$. Then
$G'(M_n)$ belongs to $\mathcal{A}_p^2{\mathcal A}\cap {\mathcal Z}{\mathcal N}_5{\mathcal A}$ and simulates $M_n$ as described
in Theorem \ref{tmm}. Let $R$ be the (finite) set of defining relations of $G'(M_n)$. Let $X$ be the set of
numbers $\epsilon$ such that $M_n$ accepts the configuration $(\epsilon,0,1,0)$. Let $X'$ be the set of numbers
$\epsilon$ such that $M_n$ works infinitely long starting with the configuration $(\epsilon,0,1,0)$. Then $X$ and
$X'$ are recursively inseparable by the choice of $M$ and $M_n$.  For any configuration $(\epsilon,0,1,0)$ of
$M_n$ consider the corresponding element $$w(\epsilon)=q_1*a_1^{(\epsilon)}*A_1*A_2*a_3*A_3*A_4.$$

Suppose $\epsilon\in X$. Then there are only finite number of computations of $\Sym(M_n)$ starting at the
configuration $(1;\epsilon,0,1,0)$. Then as in the proof of Theorem \ref{t:rhog}, there exists a finite
homomorphic image $H$ of $G'(M_n)$ where the image of $w(\epsilon)$ is not equal to 1. Hence the universal
formula $\& R\to w(\epsilon)=1$ does not hold in the finite group $H$ from $\mathcal{A}_p^2{\mathcal A}\cap {\mathcal
Z}{\mathcal N}_5{\mathcal A}$.

Now suppose that $\epsilon\not\in X$. Consider any periodic homomorphic image $H$ of $G'(M_n)$. Let $\bar t$ be
the image of $t\in G'(M_n)$ in $H$. Then there exists a number $D$ such that for every element $x\in \bar T$,
\begin{equation}\label{e456}
x*{\bar a}_3^{(D)}=x*{\bar a}_3^{(2D)}.
\end{equation}

Since $M_n$ works infinitely long starting at the configuration $(\epsilon,0,1,0)$, by Theorem \ref{tmm} the
following equality is true for some $i,k_1,k_2$:  $$w(\epsilon)=\bar x_{\vec 1, q_iA_0}*\bar a_1^{(k_1)}*\bar
a_2^{(k_2)}*\bar a_3^{(D)}*\bar A_1*\bar A_2*\bar A_3*\bar A_4.$$ Then by (\ref{e456}) and Theorem \ref{tmm}

$$\begin{array}{l}\bar w(\epsilon)=\bar x_{\vec 1, q_iA_0}*\bar a_1^{(k_1)}*\bar a_2^{(k_2)}*\bar a_3^{(2D)}*\bar
a_4^{(D)}*\bar A_1*\bar A_2*\bar A_3*\bar A_4\\ =
\bar x_{\vec 1, q_iA_0}*\bar a_1^{(k_1)}*\bar a_2^{(k_2)}*\bar a_3^{(D)}*\bar a_4^{(D)}*\bar A_1*\bar A_2*\bar
A_3*\bar A_4\\ =
\bar x_{\vec 1, q_iA_0}*\bar a_1^{(k_1)}*\bar a_2^{(k_2)}*\bar A_1*\bar A_2*\bar A_3*\bar A_4=1.\end{array}$$
since $q_iA_3A_4=0$ in $S'(M_n)$. Hence the universal formula $\& R\to w(\epsilon)=1$ holds in $H$.

Thus the set of universal formulas $\&R\to w(\epsilon)=1$ that do not hold in some finite group from
$\mathcal{A}_p^2{\mathcal A}\cap {\mathcal Z}{\mathcal N}_5{\mathcal A}$ and the set of such formulas which hold in every
periodic group are recursively inseparable.

\end{proof}

\begin{remark} Note that the universal theory of finite metabelian groups is decidable \cite{KSs}. The same is true for the set of finite groups (and any other algebraic structures of finite type) of any finitely based variety where every finitely generated group is residually finite \cite{KSs}. On the other hand, the universal theory of all finite nilpotent groups is undecidable \cite{Kh83}. The description of all (finitely based) varieties of groups where the universal theory of finite groups is decidable is currently out of reach. From Zelmanov's solution of the restricted Burnside problem \cite{Ze90, Ze91}, it immediately follows that the universal theory of finite groups in every finitely based periodic variety of groups is decidable. That result and simulations of Minsky machines in semigroups (as in Section \ref{s:2}) were used by the third author \cite{Sapir1} to obtain a complete description of all finitely based varieties of semigroups where finite semigroups have decidable universal theory.  For more information on that problem, see \cite{KSs}.
\end{remark}

\end{document}